\documentclass[12pt,a4paper,reqno]{amsart} 
\usepackage{amssymb}
\usepackage{latexsym}
\usepackage{amsmath}
\usepackage{mathrsfs}
\usepackage{euscript}
\usepackage[X2,T1]{fontenc}
\usepackage{cite}
\usepackage{calc}                   

\newcommand{\scal}[2]{\langle #1,#2\rangle}
\newcommand{\rr}[1]{\mathbf R^{#1}}
\newcommand{\rrstar}[1]{\mathbf R_*^{#1}}
\newcommand{\zz}[1]{\mathbf Z^{#1}}

\newcommand{\nm}[2]{\Vert #1\Vert _{#2}}

\newcommand{\op}{\operatorname{Op}}

\newcommand{\SG}{\operatorname{SG}}

\newcommand{\sets}[2]{\{ \, #1\, ;\, #2\, \} }
\newcommand{\ep}{\varepsilon}
\newcommand{\fy}{\varphi}
\newcommand{\cdo}{\, \cdot \, }
\newcommand{\supp}{\operatorname{supp}}

\newcommand{\eabs}[1]{\langle #1\rangle}
\newcommand{\ON}{\operatorname{ON}}

\newcommand{\vrum}{\vspace{0.1cm}}

\newcommand{\maclB}{\mathcal B}

\newcommand{\maclS}{\mathcal S}

\newcommand{\mascE}{\mathscr E}
\newcommand{\mascF}{\mathscr F}
\newcommand{\mascH}{\mathscr H}
\newcommand{\mascI}{\mathscr I}
\newcommand{\mascP}{\mathscr P}
\newcommand{\mascS}{\mathscr S}

\newcommand{\mabfj}{\boldsymbol j}
\newcommand{\mabfk}{\boldsymbol k}
\newcommand{\mabfp}{{\boldsymbol p}}
\newcommand{\mabfq}{\boldsymbol q}

\newcommand{\splM}{\EuScript M}
\newcommand{\splW}{\EuScript W}

\numberwithin{equation}{section}          

\newtheorem{thm}{Theorem}
\numberwithin{thm}{section}

\newcommand{\rubrik}{}
\newtheorem{prop}[thm]{Proposition}
\newtheorem{cor}[thm]{Corollary}
\newtheorem{lemma}[thm]{Lemma}

\theoremstyle{definition}

\newtheorem{defn}[thm]{Definition}

\theoremstyle{remark}
\newtheorem{rem}[thm]{Remark}              

\title{Continuity and compactness for
pseudo-differential operators with symbols in
quasi-Banach spaces or H{\"o}rmander classes}

\author{Joachim Toft}

\address{Department of Mathematics,
Linn{\ae}us University, V{\"a}xj{\"o}, Sweden}

\email{joachim.toft@lnu.se}



\begin{document}

\par

\begin{abstract}
We deduce continuity and Schatten-von Neumann properties for
operators
with matrices satisfying mixed quasi-norm estimates with Lebesgue
and Schatten parameters in $(0,\infty ]$. We use these results to deduce
continuity and Schatten-von Neumann properties for pseudo-differential
operators with symbols in quasi-Banach modulation
spaces, or in appropriate H{\"o}rmander classes.
\end{abstract}

\maketitle


\par

\section{Introduction}\label{sec0}

\par

The singular values for a linear operator is a non-increasing sequence of
non-negative numbers which are strongly
linked to questions on continuity and compactness for the operator
in the following sense:
\begin{itemize}
\item the operator is continuous, if and only if its singular values are bounded.

\vrum

\item the operator is compact, if and only if its singular values decay towards
zero at infinity. 
Moreover, fast decays of the singular values permit efficient finite
rank approximations.
%

\vrum

\item the operator has rank $j\ge 0$, if and only if its singular values of
order $j+1$ and higher are zero.
\end{itemize}
(See \cite{Ho1,Gc2} and Section \ref{sec1} for definitions.)

\par

%

In particular, there is a strong connection between the decay of the
singular values and finding pseudo-inverses in convenient ways, since
such questions are linked to efficient finite rank approximations.

\par

One way to measure the decay of singular values is to consider
Schatten-von Neumann classes. More precisely, let $T$ be a linear
operator. Then
$T$ belongs to $\mascI _p$, the set of Schatten-von Neumann
operators of order $p\in (0,\infty ]$, if and only if its singular values
$\sigma _1(T),\sigma _2(T),\dots$  belong to $\ell ^p$. Since the
singular values are non-negative and non-increasing, it follows that
\begin{equation}\label{SiVaSchattProp}
\begin{alignedat}{3}
\sigma _j(T) &= o (j^{-1/p}),&
\quad &\text{when} &\quad
T&\in \mascI _p,\ p<\infty ,
\\[1ex]
\sigma _j(T) &\neq o (j^{-1/p}),&
\quad &\text{when} &\quad
T&\notin \mascI _{p+\ep},\ p<\infty ,\ \ep >0,
\end{alignedat}
\end{equation}
which indicates the link
between Schatten-von Neumann classes and the decays of
singular values.

\par

It is in general a difficult task to find exact and convenient characterizations of
Schatten-von Neumann classes. One is therefore left to find suitable
necessary or sufficient conditions when characterizing such classes. For
example a Toeplitz operator, acting on $L^2$ belongs to $\mascI _p$,
$p\in [1,\infty]$, when its symbol belongs to $L^p$ (cf. \cite{HeWo,BaCoIs,BaIs}).
For pseudo-differential operators $\op (a)$, the situation is slightly different
since $\op (a)$ might not be in $\mascI _p$, $p\neq 2$,
when its symbol $a$ belongs to $L^p$. On the other hand, by adding
further restrictions on the symbols it is possible to deduce similar
sufficient conditions as for Toeplitz operators. For example,
if $S(m,g)$ is an appropriate H{\"o}rmander class parameterized
with the Riemannian metric $g$ and weight function $m$ on the
phase space, then
\begin{equation}\label{HormSchattenEquiv}
\sets {\op _t(a)}{a\in S(m,g)}\subseteq \mascI _p
\qquad \Longleftrightarrow \qquad
m\in L^p.
\end{equation}
(Cf. Theorems 2.1 and 2.9 in \cite{BuTo}. See also \cite{Ho0,Ho1,Toft4}
for pre-results.)

\par

There are several Schatten-von Neumann results for pseudo-differential
operators with symbols in modulation spaces, Besov spaces and Sobolev
spaces (cf. \cite{Toft5} and the references therein). In particular,
let $M^{p,q}$ be the classical modulation space with parameters
$p,q\in [1,\infty ]$, introduced by Feichtinger in \cite{Fei1}. Then
\begin{gather}
\op (a)  \in \mascI _p 
\quad  \text{when} \quad
a\in M^{p,p}, \ p\in [1,2],\label{ModSchatten1}
\\[1ex]
\op (a) \, :\, M^{p_1,q_1}\to M^{p_1,q_1} 
\quad  \text{when} \quad
a\in M^{\infty ,1},\ p_1,q_1\in [1,\infty ],
\label{ModCont1}
\intertext{and}
\op (a) \, :\, M^{\infty ,\infty }\to M^{1,1} 
\quad  \text{when} \quad
a\in M^{1 ,1}.
\label{ModCont1B}
\end{gather}
The relation
\eqref{ModSchatten1} was essentially deduced by
Gr{\"o}chenig and Heil, although it seems to be well-known
earlier by Feichtinger (cf. \cite[Proposition 4.1]{GH1}). The
relation \eqref{ModCont1} was first proved in \cite{Gc2},
with certain pre-results given already in \cite{GH1,Sj1}, and
\eqref{ModCont1B} is in some sense obtained by Feichtinger
already in \cite{Fei1}.

\par

There are also several extensions and modifications of these
results. For example, in \cite{GH2,Toft2} it was proved that
\begin{equation}\label{ModCont1C}
\begin{gathered}
\op (a)\, :\, M^{p_1,q_1}\to M^{p_2,q_2}
\quad \text{when} \quad
a\in M^{p,q}, \ q\le \min (p,p')
\\[1ex]
\quad \text{and} \quad
\frac 1{p_1} - \frac 1{p_2} = \frac 1{q_1} - \frac 1{q_2} = 1-\frac 1{p} - \frac 1{q},
\ q\le p_2,q_2\le p,
\end{gathered}
\end{equation}
which covers both \eqref{ModCont1} and \eqref{ModCont1B}. See also
\cite{Toft3,Toft5,Toft8,Toft11} for extensions of the latter result to weighted spaces,
and \cite{MoPf,Wa} for related results with other types of modulation spaces
as symbol classes. Furthermore, in \cite{DeRu1,DeRu2,DeRu3,FiRu},
related analysis in background of compact or local-compact Lie groups can
be found. 

\par

In the literature, it is usually assumed that $p$ and $q$ here above belong
to $[1,\infty ]$ instead of the larger interval $(0,\infty ]$. An important reason
for excluding the cases $p<1$ or $q<1$ is that the involved spaces fail
to be local convex, leading in general to several additional difficulties
compared to the situation when $p,q\in [1,\infty ]$.
On
the other hand, in view of \eqref{SiVaSchattProp} it is valuable to decide
whether an operator belongs to $\mascI _p$ or not, also in the case $p<1$.
Here we remark that convenient Schatten-$p$ results with $p<1$ can be found
for Hankel and Toeplitz operators in e.{\,}g. \cite{Is}, and for pseudo-differential operators
on compact Lie groups in e.{\,}g. \cite{DeRu1,DeRu2,DeRu3}.

\medspace

In the paper we deduce weighted extensions of
\eqref{HormSchattenEquiv}--\eqref{ModCont1B}, where in contrast to
\cite{Toft3,Toft5,Toft8,Toft11}, the case $p<1$
is included. First we deduce continuity and Schatten-von
Neumann properties for suitable types of matrix operators. Thereafter we
carry over these results to the case of pseudo-differential operators
with symbols in modulation spaces, using Gabor analysis as link, in
analogous ways as in e.{\,}g. \cite{GH1,Gc3,Gc4,GrSt,Toft5,WaSc}.

\par

Here we remark that our analysis is comprehensive
compared to \cite{GH1,Gc3,Gc4,GrSt,Toft5,WaSc} 
because of the absent of local-convexity.
The situation is handled by using the Gabor analysis in \cite{GaSa,Toft12},
for non-local convex modulation spaces, in combination of suitable
factorization techniques for matrix operators.

\par

In order to shed some more light we explain
some consequences of our investigations. As a special case of
Theorem \ref{thmOpSchatten} we have 
\begin{equation}\tag*{(\ref{ModSchatten1})$'$}
\op (a) \in \mascI _p
\quad \text{when} \quad
a\in M^{p,p}, \ p\in (0,2],
\end{equation}
i.{\,}e. \eqref{ModSchatten1} still holds after
$[1,2]$ is replaced by the larger interval $(0,2]$. Furthermore,
we prove that \eqref{ModSchatten1}$'$ is sharp in
the sense that any
modulation space (with trivial weight), and not contained in
$M^{p,p}(\rr {2d})$, contains symbols, whose corresponding
pseudo-differential operators fail to belong to $\mascI _p$ (cf.
Theorem \ref{SchattenConverse}).

\par

In Section \ref{sec3} we also deduce general continuity results for
pseudo-differential operators with symbols in modulation spaces. 
In particular, \eqref{ModCont1C} is extended in Theorem
\ref{thmOpCont} extend in several ways, and as special case,
\eqref{ModCont1} and \eqref{ModCont1B} are extended into
\begin{multline}\tag*{(\ref{ModCont1})$'$}
\op (a) \, :\, M^{p_1,q_1}\to M^{p_1,q_1}
\\[1ex]
\text{when} \quad
a\in M^{\infty ,q},\ p_1,q_1\in [q,\infty ],\ q\in (0,1],
\end{multline}
and
\begin{equation}\tag*{(\ref{ModCont1B})$'$}
\op (a) \, :\, M^{\infty ,\infty }\to M^{q,q} 
\quad  \text{when} \quad
a\in M^{q ,q},\ q\in (0,1].
\end{equation}

\par

In Section \ref{sec4} we apply \eqref{ModSchatten1}$'$
to deduce Schatten-von Neumann
properties for pseudo-differential operators with symbols in $S(m,g)$ in
H{\"o}rmander-Weyl calculus. In particular we show that the sufficiency
part of \eqref{HormSchattenEquiv} still holds for $p\in (0,1]$
(cf. Theorem \ref{thm:WeylHorm1}). That is, for suitable $m$ and $g$, we have
\begin{equation*}
\sets {\op _t(a)}{a\in S(m,g)}\subseteq \mascI _p
\quad \text{when} \quad
m\in L^p.
\end{equation*}

\medspace

An important part behind the analysis
concerns Theorem \ref{factorizationprop}, which in the
non-weighted case,
essentially state that any matrix
$A\in \mathbb U^{p_0}$ can be factorized as
\begin{equation}\label{MatrixFaktor0}
A=A_1\cdot A_2,
\quad \text{when}\quad
A_j\in \mathbb U^{p_j},\quad
\frac 1{p_1}+\frac 1{p_2}= \frac 1{p_0}.
\end{equation}
From these relations we obtain
\begin{equation}\label{embSimpleCase}
\mathbb U^{p}\subseteq \mascI _{p},\quad
\text{when}\quad p\in (0,2].
\end{equation}

\par

In fact, the set of Hilbert-Schmidt operators on $\ell ^2$
agrees with $\mathbb U^2$, and with
$\mascI _2$ (also in norms).  Consequently, $\mathbb U^{2}
=\mascI _2$, and H{\"o}lder's inequality for
Schatten-von Neumann classes together with \eqref{MatrixFaktor0}
give that for every $A\in \mathbb U^{2/N}$,
with integer $N\ge 1$, there are matrices
$A_1,\dots ,A_N\in \mathbb U^{2}$ such that
\begin{equation*}
A = A_1\cdots A_N \in \mathbb U^{2}\circ \cdots \circ \mathbb U^{2}
= \mascI _2\circ \cdots \circ \mascI _2 = \mascI _{2/N}.
\end{equation*}
Hence $\mathbb U^{2/N}\subseteq \mascI _{2/N}$ for every
integer $N\ge 1$. 
A (real) interpolation argument between the cases
$$
\mathbb U^{2/N}\subseteq \mascI _{2/N}\quad
\text{and}\quad
\mathbb U^{2} =  \mascI _{2}
$$
now shows that that $\mathbb U^{p}\subseteq \mascI _{p}$
when $p\in [2/N,2]$. Since $2/N$ can be chosen arbitrarily
close to $0$, \eqref{embSimpleCase} follows.

\par

In Section \ref{sec2}, the previous arguments are used to deduce
more general versions of \eqref{embSimpleCase} involving weighted spaces.
(See Theorem \ref{MatrixSchatten}.)

\par

In Section \ref{sec5} we show some applications and other results for Schatten-von
Neumann symbols. Here we introduce the set $s_{t,p}^q$ consisting of all
symbols $a$ such that $\op _t(a)$ belongs to $\mascI _p$ and such that
the orthonormal sequences of the eigenfunctions to $|\op _t(a)|$ and $|\op _t(a)^*|$
are bounded sets in the modulation space $M^{2q}$. It follows that $s_{t,p}^q$
is contained in $s_{t,p}$ the set of all symbols $a$ such that $\op _t(a)\in \mascI _p$. 

\par

We prove that $\mascS$ is continuously embedded in $s_{t,p}^q$, and that
$$
s_{t,p}^p \cap \mascE ' \subseteq \mascF L^p\cap \mascE ' \subseteq s_{t,p} \cap \mascE ' 
$$
for every $p>0$.

\par

Finally we remark
that in \cite{DeRu1,DeRu2,DeRu3,FiRu},
Delgado, Fischer, Ruzhansky and Turunen deal with various kinds
of continuity and compactness questions for pseudo-differential
operators acting on functions defined on suitable Lie groups.
In their approach, matrix-valued
symbols appear naturally, and several interesting results on
matrices are deduced. A part of these investigations are related to
the analysis in Section \ref{sec2}.

\par

\section*{Acknowledgement}

\par

The author is very grateful to K. H. Gr{\"o}chenig, for valuable
advices and comments, leading
to several improvements of the content and the style. He is also
grateful to P. Wahlberg for valuable comments.

\par

\section{Preliminaries}\label{sec1}

\par

In this section we recall some facts on Gelfand-Shilov spaces,
modulation spaces and Schatten-von Neumann classes.
The proofs are in general omitted.

\par

\subsection{Weight functions}

\par

We start by discussing general properties on weight
functions. A \emph{weight} on $\rr d$ is a positive function $\omega
\in  L^\infty _{loc}(\rr d)$ such that $1/\omega \in  L^\infty _{loc}(\rr d)$.
We usually assume that $\omega$ is \emph{moderate},
or \emph{$v$-moderate} for some positive function $v \in
 L^\infty _{loc}(\rr d)$. This means that
\begin{equation}\label{moderate}
\omega (x+y) \lesssim \omega (x)v(y),\qquad x,y\in \rr d.
\end{equation}
Here $A\lesssim B$ means that $A\le cB$
for a suitable constant $c>0$, and for future references, we
write $A\asymp B$
when $A\lesssim B$ and $B\lesssim A$. 
We note that \eqref{moderate} implies that $\omega$ fulfills
the estimates
\begin{equation}\label{moderateconseq}
v(-x)^{-1}\lesssim \omega (x)\lesssim v(x),\quad x\in \rr d.
\end{equation}
Furthermore, if $v$ in \eqref{moderate} can be chosen as a polynomial,
then $\omega$ is called a weight of \emph{polynomial type}. We let
$\mascP (\rr d)$ and $\mascP _E(\rr d)$ be the sets of all weights of
polynomial type and moderate weights on $\rr d$, respectively.

\par

It can be proved that if $\omega \in \mascP _E(\rr d)$, then
$\omega$ is $v$-moderate for some $v(x) = e^{r|x|}$, provided the
positive constant $r>0$ is chosen large enough (cf. \cite{Gc2.5}). In particular,
\eqref{moderateconseq} shows that for any $\omega \in \mascP
_E(\rr d)$, there is a constant $r>0$ such that
\begin{equation}\label{WeightExpEst}
e^{-r|x|}\lesssim \omega (x)\lesssim e^{r|x|},\quad x\in \rr d
\end{equation}
(cf. \cite{Gc2.5}).

\par

We say that $v$ is
\emph{submultiplicative} if $v$ is even and \eqref{moderate}
holds with $\omega =v$. In the sequel, $v$ and $v_j$ for
$j\ge 0$, always stand for submultiplicative weights if
nothing else is stated.

\par

\subsection{Gelfand-Shilov spaces}

\par

Next we recall the definition of Gelfand-Shilov spaces.

\par

Let $h,s\in \mathbf R_+$ be fixed. Then $\mathcal S_{s,h}(\rr d)$
is the set of all $f\in C^\infty (\rr d)$ such that
\begin{equation*}
\nm f{\mathcal S_{s,h}}\equiv \sup \frac {|x^\beta \partial ^\alpha
f(x)|}{h^{|\alpha + \beta |}(\alpha !\, \beta !)^s}
\end{equation*}
is finite. Here the supremum is taken over all $\alpha ,\beta \in
\mathbf N^d$ and $x\in \rr d$.

\par

Obviously $\mathcal S_{s,h}(\rr d)$ is a Banach space which increases as $h$
and $s$ increase,
and is contained in $\mascS (\rr d)$, the set of Schwartz functions on $\rr d$.
Furthermore, if $s>1/2$, or $s=1/2$ and $h$ is sufficiently large, then
is dense in $\mathscr S$. Hence, the dual $(\mathcal S_{s,h})'(\rr d)$ of
$\mathcal S_{s,h}(\rr d)$ is a Banach space which contains $\mathscr S'(\rr d)$.

\par

The \emph{Gelfand-Shilov spaces} $\mathcal S_{s}(\rr d)$ and
$\Sigma _s(\rr d)$ are the inductive and projective limits respectively
of $\mathcal S_{s,h}(\rr d)$ with respect to $h$. This implies that
\begin{equation}\label{GSspacecond1}
\mathcal S_s(\rr d) = \bigcup _{h>0}\mathcal S_{s,h}(\rr d)
\quad \text{and}\quad \Sigma _{s}(\rr d) =\bigcap _{h>0}\mathcal
S_{s,h}(\rr d),
\end{equation}
and that the topology for $\mathcal S_s(\rr d)$ is the
strongest possible one such that each inclusion map
from $\mathcal S_{s,h}(\rr d)$ to $\mathcal S_s(\rr d)$
is continuous. The space $\Sigma _s(\rr d)$ is a Fr{\'e}chet
space with semi norms $\nm \cdo{\mathcal S_{s,h}}$, $h>0$.
Moreover, $\mathcal S _s(\rr d)\neq \{ 0\}$, if and only if
$s\ge 1/2$, and $\Sigma _s(\rr d)\neq \{ 0\}$,
if and only if $s>1/2$.

\par

For every $\ep >0$ and $s>0$, we have
$$
\Sigma _s (\rr d)\subseteq \mathcal S_s(\rr d)\subseteq
\Sigma _{s+\ep}(\rr d).
$$

\par

%

\medspace

The \emph{Gelfand-Shilov distribution spaces} $\mathcal S_s'(\rr d)$
and $\Sigma _s'(\rr d)$ are the projective and inductive limit
respectively of $\mathcal S_{s,h}'(\rr d)$.  Hence
\begin{equation}\tag*{(\ref{GSspacecond1})$'$}
\mathcal S_s'(\rr d) = \bigcap _{h>0}\mathcal
S_{s,h}'(\rr d)\quad \text{and}\quad \Sigma _s'(\rr d)
=\bigcup _{h>0} \mathcal S_{s,h}'(\rr d).
\end{equation}
By \cite{Pil}, $\mathcal S_s'$ and $\Sigma _s'$
are the duals of $\mathcal S_s$ and $\Sigma _s$, 
respectively.

\par

The Gelfand-Shilov spaces are invariant or posses convenient
mapping properties under several basic
transformations. For example they are invariant under
translations, dilations, and under (partial) Fourier transformations.

\par

From now on we let $\mathscr F$ be the Fourier transform,
given by
$$
(\mathscr Ff)(\xi )= \widehat f(\xi ) \equiv (2\pi )^{-d/2}\int _{\rr
{d}} f(x)e^{-i\scal  x\xi }\, dx
$$
when $f\in L^1(\rr d)$. Here $\scal \cdo \cdo$ denotes the
usual scalar product on $\rr d$. The map $\mathscr F$ extends 
uniquely to homeomorphisms on $\mathscr S'(\rr d)$, $\mathcal
S_s'(\rr d)$ and $\Sigma _s'(\rr d)$, and restricts to 
homeomorphisms on $\mathscr S(\rr d)$, $\mathcal S_s(\rr d)$
and $\Sigma _s(\rr d)$, and to a unitary operator on $L^2(\rr d)$.

\medspace

Next we recall some mapping properties of Gelfand-Shilov
spaces under short-time Fourier transforms.
Let $\phi \in \mathscr S(\rr d)$ be fixed. For every $f\in
\mathscr S'(\rr d)$, the \emph{short-time Fourier transform} $V_\phi
f$ is the distribution on $\rr {2d}$ defined by the formula
\begin{equation}\label{defstft}
(V_\phi f)(x,\xi ) =\mathscr F(f\, \overline{\phi (\cdo -x)})(\xi ) =
(f,\phi (\cdo -x)e^{i\scal \cdo \xi}).
\end{equation}
We recall that if $T(f,\phi )\equiv V_\phi f$ when $f,\phi \in \maclS _{1/2}(\rr d)$,
then $T$ is uniquely extendable to sequentially continuous mappings
\begin{alignat*}{2}
T\, &:\, & \maclS _s'(\rr d)\times \maclS _s(\rr d) &\to
\maclS _s '(\rr {2d})\bigcap C^\infty (\rr {2d}),
\\[1ex]
T\, &:\, & \maclS _s'(\rr d)\times \maclS _s'(\rr d) &\to
\maclS _s '(\rr {2d}),
\end{alignat*}
and similarly when $\maclS _s$ and $\maclS _s'$ are replaced
by $\Sigma _s$ and $\Sigma _s'$, respectively, or by
$\mascS$ and $\mascS '$, respectively (cf. \cite{CPRT10,Toft8}).
We also note that $V_\phi f$ takes the form
\begin{equation}\tag*{(\ref{defstft})$'$}
V_\phi f(x,\xi ) =(2\pi )^{-d/2}\int _{\rr d}f(y)\overline {\phi
(y-x)}e^{-i\scal y\xi}\, dy
\end{equation}
when $f\in L^p_{(\omega )}(\rr d)$ for some $\omega \in
\mascP _E(\rr d)$, $\phi \in \Sigma _1(\rr d)$ and $p\ge 1$. Here
$L^p_{(\omega )}(\rr d)$, when $p\in (0,\infty ]$ and
$\omega \in \mascP _E(\rr d)$, is the set of all $f\in L^p_{loc} (\rr d)$ such
that $f\cdot \omega \in L^p(\rr d)$. 

\par

\subsection{Mixed quasi-normed space of Lebesgue types}

\par

Let $p,q\in (0,\infty ]$, and let $\omega \in \mascP _E(\rr {2d})$.
Then $L^{p,q}_{(\omega )}(\rr {2d})$ and $L^{p,q}_{*,(\omega )}(\rr {2d})$ consist
of all measurable functions $F$ on $\rr {2d}$ such that
\begin{alignat*}{3}
\nm {g_1}{L^q(\rr d)}&<\infty ,&
\quad &\text{where} &\quad
g_1(\xi ) &\equiv \nm {F(\cdo ,\xi )\omega (\cdo ,\xi )}{L^p(\rr d)}
\intertext{and}
\nm {g_2}{L^p(\rr d)}&<\infty ,&
\quad &\text{where} &\quad
g_2(x) &\equiv \nm {F(x,\cdo )\omega (x,\cdo )}{L^q(\rr d)},
\end{alignat*}
respectively.

\par

More generally, let
$$
\mabfp =(p_1,\dots , p_d)\in (0,\infty ]^d,
\quad 
\mabfq =(q_1,\dots , q_d)\in (0,\infty ]^d,
$$
$\operatorname {S}_d$ be the set of permutations
on $\{ 1,\dots ,d\}$, $\mabfp \in
(0,\infty ]^d$, $\omega \in \mascP _E(\rr d)$, and let $\sigma
\in \operatorname {S}_d$. Moreover, let $\Omega _j\subseteq
\mathbf R$ be Borel-sets, $\mu _j$ be positive Borel
measures on $\Omega _j$, $j=1,\dots ,d$, and let
$\Omega =\Omega _1\times \cdots \times \Omega _d$
and $\mu = \mu _1\otimes \cdots \otimes \mu _d$.
For every measurable and complex-valued function $f$ on
$\Omega$, let
$g_{j,\omega ,\mu}$, $j=1,\dots ,d-1$, be defined inductively by
\begin{align*}
g_{0,\omega ,\mu}(x_1,\dots ,x_d)
&\equiv |f (x_{\sigma ^{-1}(1)},\dots ,x_{\sigma ^{-1}(d)})
\omega (x_{\sigma ^{-1}(1)},\dots ,x_{\sigma ^{-1}(d)})|,
\\[1ex]
g_{k,\omega ,\mu}(x_{k+1},\dots ,x_d) &\equiv \nm {g_{k-1,\omega ,\mu}(\cdo ,
x_{k+1},\dots ,x_d) }
{L^{p_k}(\mu _k)},
\quad k=1,\dots ,d-1 ,
\intertext{and let }
\nm f{L^{\mabfp}_{\sigma ,(\omega )}(\mu)} &\equiv
\nm {g_{d-1,\omega ,\mu}}{L^{p_d}(\mu _d)}.
\end{align*}
The mixed quasi-norm space $L^{\mabfp}_{\sigma ,(\omega )}(\mu)$ of
Lebesgue type is defined as the set of all $\mu$-measurable functions
$f$ such that $\nm f{L^{\mabfp }_{\sigma ,(\omega )}(\mu )}<\infty$.

\par

In the sequel we have $\Omega =\rr d$ and $d\mu = dx$, or
$\Omega =\Lambda$ and $\mu (j)=1$ when $j \in \Lambda$, where
\begin{equation}\label{LambdaDef}
\begin{aligned}
\Lambda &= \Lambda _{[\theta ]} = T_\theta \zz d \equiv
\sets {(\theta _1j_1,\dots ,\theta _dj_d)}{(j_1,\dots ,j_d)\in \zz d} ,
\\[1ex]
\theta &=(\theta _1,\dots ,\theta _d)\in \rrstar d,
\end{aligned}
\end{equation}
and $T_\theta$ denotes the diagonal matrix with diagonal elements
$\theta _1,\dots ,\theta _d$. In the former case we set
$L^{\mabfp}_{\sigma ,(\omega )}(\mu)=L^{\mabfp}_{\sigma ,(\omega )}=
L^{\mabfp}_{\sigma ,(\omega )}(\rr d)$, and in the latter
case we set $L^{\mabfp}_{\sigma ,(\omega )}(\mu) =
\ell ^{\mabfp}_{\sigma ,(\omega )}(\Lambda )$.

\par

For conveniency we also set $L^{\mabfp }_{(\omega )}=L^{\mabfp }_{\sigma ,(\omega )}$
and $\ell ^{\mabfp } _{(\omega )}=\ell ^{\mabfp } _{\sigma ,(\omega )}$ when $\sigma$
is the identity map, and we let $\ell (\Lambda )$ be the set of all
(complex-valued) sequences on $\Lambda$ and $\ell _0
(\Lambda )$ be the set of all $f\in \ell (\Lambda )$
such that $f(j)\neq 0$ for at most finite numbers of $j$.
Furthermore, if $\omega$ is equal to $1$, then we set
\begin{equation*}
L^{\mabfp }_{\sigma } =L^{\mabfp }_{\sigma ,(\omega )},
\quad
\ell ^{\mabfp } _{\sigma }=\ell ^{\mabfp } _{\sigma ,(\omega )},
\quad
L^{\mabfp } = L^{\mabfp }_{(\omega )}
\quad \text{and}  \quad
\ell ^{\mabfp } = \ell ^{\mabfp } _{(\omega )}.
\end{equation*}

\par

Let $\mabfp = (p_1,\dots ,p_d)\in [0,\infty ]^d$,
$\mabfq=(q_1,\dots ,q_d)\in [0 ,\infty ]^d$ and $t\in [-\infty ,\infty ]$. Then we use
the conventions
$$
\mabfp \le \mabfq \quad \text{and}\quad \mabfp \le t \quad \text{when}\quad
p_j\le q_j\ \text{and}\ p_j\le t,
$$
respectively, for every $j=1,\dots ,d$, and
$$
\mabfp = \mabfq \quad \text{and}\quad \mabfp = t \quad \text{when}\quad
p_j= q_j\ \text{and}\ p_j= t,
$$
respectively, for every $j=1,\dots ,d$. The relations
$\mabfp < \mabfq$ and $\mabfp < t$ are defined analogously. We also let
$$
\mabfp \pm \mabfq=(p_1\pm q_1,\dots ,p_d\pm q_d)
\quad \text{and}\quad
\mabfp \pm t=(p_1\pm t,\dots ,p_d\pm t),
$$
provided the right-hand sides are well-defined and belongs to $[-\infty ,\infty ]^d$.
Moreover, we set $1/0=\infty$, $1/\infty =0$ and $1/\mabfp =(1/p_1,\dots ,1/p_d)$.

\par

We also let
$$
\max (\mabfp ) \equiv \max (p_1,\dots ,p_d)
\quad \text{and}\quad
\min (\mabfp ) \equiv \min (p_1,\dots ,p_d),
$$
and note that if $\max (\mabfp ) <\infty$, then $\ell _0 (\Lambda )$ is dense in
$\ell ^\mabfp _{\sigma ,(\omega )} (\Lambda )$.

\par

\subsection{Modulation spaces}\label{subsec1.2}

\par

Next we define modulation spaces.
Let $\phi \in \maclS _{1/2}(\rr d)\setminus 0$. For any $p,q\in (0.\infty ]$
and $\omega \in \mascP _E(\rr {2d})$,
the modulation spaces $M^{p,q}_{(\omega )}(\rr d)$ and $W^{p,q}_{(\omega )}(\rr d)$
are the sets of all $f\in \maclS _{1/2}'(\rr d)$ such that $V_\phi f\in
L^{p,q}_{(\omega )}(\rr {2d})$ and $V_\phi f\in
L^{p,q}_{*,(\omega )}(\rr {2d})$, respectively. We equip these spaces
with the quasi-norms
$$
\nm f{M^{p,q}_{(\omega )}}\equiv \nm {V_\phi f}{L^{p,q}_{(\omega )}}
\quad \text{and}\quad
\nm f{W^{p,q}_{(\omega )}}\equiv \nm {V_\phi f}{L^{p,q}_{*,(\omega )}},
$$
respectively.
One of the most common types of modulation spaces concerns
$M^{p,q}_{(\omega )}(\rr d)$ with $\omega \in \mascP (\rr {2d})$,
and are sometimes called standard modulation spaces.
They were introduced by Feichtinger in \cite{Fei1} for certain
choices of $\omega$.

\par

More generally, for any $\sigma \in \operatorname S_{2d}$,
$\mabfp \in (0,\infty ]^{2d}$ and $\omega \in \mascP _E(\rr {2d})$,
the modulation space $M^\mabfp _{\sigma ,(\omega )}(\rr d)$
is the set of all $f\in \maclS _{1/2}'(\rr d)$ such that $V_\phi f\in
L^{\mabfp}_{\sigma ,(\omega )}(\rr {2d})$, and we equip
$M^{\mabfp}_{\sigma (\omega )}(\rr d)$ with the quasi-norm
\begin{equation}\label{modnorm2}
\nm f{M^{\mabfp}_{\sigma , (\omega )} }\equiv
\nm {V_\phi f}{L^{\mabfp}_{\sigma ,(\omega )}}.
\end{equation}

\par

For conveniency we set $M^p_{(\omega )}=M^{p,p}_{(\omega)}$, and if
$\omega =1$ everywhere, then set
$$
M^\mabfp =M^\mabfp _{\sigma ,(\omega )},\quad
M^{p,q} = M^{p,q}_{(\omega )},\quad
W^{p,q} = W^{p,q}_{(\omega )}
\quad \text{and}\quad
M^{p} = M^{p}_{(\omega )}.
$$

\par

In the following propositions we list some properties for modulation
spaces,
and refer to \cite{Fei1,FG1,Gc2,Toft5} for proofs.

\par

\begin{prop}\label{p1.4A}
Let $\sigma \in \operatorname S_{2d}$ and
$\mabfp \in (0,\infty ]^{2d}$. Then the following is true:
\begin{enumerate}
\item[{\rm{(1)}}] if $\omega \in \mascP  _{E}(\rr {2d})$, then $\Sigma _1(\rr d)
\subseteq M^{\mabfp}_{\sigma ,(\omega )}(\rr d) \subseteq \Sigma _1'(\rr d)$;

\vrum

\item[{\rm{(2)}}] if $\omega \in \mascP  _{E}(\rr {2d})$ satisfies \eqref{WeightExpEst}
for every $r >0$, then
$\maclS _1(\rr d)\subseteq M^{\mabfp}_{\sigma ,(\omega )}(\rr d)
\subseteq \maclS _1'(\rr d)$;

\vrum

\item[{\rm{(3)}}]  if $\omega \in \mascP (\rr {2d})$, then
$\mathscr S(\rr d)\subseteq M^{\mabfp}_{\sigma ,(\omega )}(\rr d)
\subseteq \mathscr S '(\rr d)$.
\end{enumerate}
\end{prop}

\par

\begin{prop}\label{p1.4B}
Let $\sigma \in \operatorname S_{2d}$, $r\in (0,1]$, $\mabfp ,\mabfp _j\in
(0,\infty ] ^{2d}$ and $\omega ,\omega _j,v\in \mascP  _{E}(\rr {2d})$, $j=1,2$,
be such that $r\le \mabfp$,
$\mabfp _1\le \mabfp _2$,  $\omega _2\lesssim \omega _1$, and
$\omega$ is $v$-moderate. Then the following is true:
\begin{enumerate}
\item if $\phi \in M^r_{(v)}(\rr d)\setminus 0$, then
$f\in M^{\mabfp}_{\sigma ,(\omega )}(\rr d)$, if and only if
\eqref{modnorm2} is finite.
In particular, $M^{\mabfp}_{\sigma ,(\omega )}(\rr d)$ is independent
of the choice of $\phi \in M^r_{(v)}(\rr d)\setminus 0$.
Moreover, $M^{\mabfp}_{\sigma ,(\omega )}(\rr d)$ is a quasi-Banach
space under the quasi-norm in \eqref{modnorm2}, and different
choices of $\phi$ give rise to equivalent quasi-norms.

\par

If in addition $\mabfp \ge 1$, then $M^{\mabfp}_{\sigma ,(\omega )}
(\rr d)$ is a Banach space with norm \eqref{modnorm2};

\vrum

\item[{\rm{(2)}}] $M^{\mabfp _1}_{\sigma ,(\omega _1)}(\rr d)\subseteq
M^{\mabfp _2}_{\sigma ,(\omega _2)}(\rr d)$.
\end{enumerate}
\end{prop}

\par

Next we discuss Gabor expansions, and start by recalling some notions.
It follows from the analysis in Chapters 11--14 in \cite{Gc2}
that the operators in the following definition are well-defined and
continuous.

\par

\begin{defn}\label{DefAnSynGabOps}
Let $\Lambda =\Lambda _{[\theta ]}$ be as in
\eqref{LambdaDef}, $\omega ,v\in \mascP _E(\rr {2d})$ be such that
$\omega$ is $v$-moderate, and let $\phi ,\psi \in M^1_{(v)}(\rr d)$.
\begin{enumerate}
\item The \emph{analysis operator} $C^{\Lambda}_\phi$ is the operator from
$M^\infty _{(\omega )}(\rr d)$ to $\ell ^{\infty}_{(\omega)}(\Lambda )$,
given by
$$
C^\Lambda _\phi f \equiv \{ V_\phi f(j,\iota ) \} _{j,\iota \in \Lambda} \text ;
$$

\vrum

\item The \emph{synthesis operator} $D^{\Lambda}_\psi$ is the operator from
$\ell ^\infty _{(\omega )}(\Lambda )$ to $M^\infty _{(\omega)}(\rr d)$,
given by
$$
D^\Lambda _\psi c \equiv \sum _{j,\iota \in \Lambda} c_{j,\iota}
e^{i\scal \cdo {\iota }}\phi (\cdo -j)\text ;
$$

\vrum

\item The \emph{Gabor frame operator} $S^{\Lambda}_{\phi ,\psi}$
is the operator on $M^\infty _{(\omega )}(\rr d)$,
given by $D^\Lambda _\psi \circ
C^\Lambda _\phi$, i.{\,}e.
$$
S^{\Lambda}_{\phi ,\psi}f \equiv \sum _{j,\iota \in \Lambda} V_\phi f(j,\iota )
e^{i\scal \cdo {\iota }}\psi (\cdo -j).
$$
\end{enumerate}
\end{defn}

\par

We usually assume that
$\theta _1= \cdots =\theta _d=\ep >0$, and then we set
$\Lambda _\ep = \Lambda _{[\theta ]}$.

\par

The proof of the following result is omitted since the result follows from
Theorem 13.1.1 and other results in \cite{Gc2} (see also Theorem S in
\cite{Gc1}).

\par

\begin{prop}\label{ThmS}
Let $\Lambda$ be as in \eqref{LambdaDef},
$v\in \mascP _E(\rr {2d})$ be submultiplicative, and $\phi \in
M^1_{(v)}(\rr d)\setminus 0$.
Then the following is true:
\begin{enumerate}
\item if
\begin{equation}\label{DualFrames}
\{ e^{i\scal {\cdo }{\iota }}\phi (\cdo -j) \} _{j,\iota \in \Lambda}
\quad \text{and}\quad
\{ e^{i\scal {\cdo }{\iota }}\psi (\cdo -j) \} _{j,\iota \in \Lambda}
\end{equation}
are dual frames to each others, then $\psi \in M^1_{(v)}(\rr d)$;

\vrum

\item there is a constant $\ep _0>0$ such that the frame operator
$S_{\phi ,\phi}^\Lambda$ is a homeomorphism on $M^1_{(v)}(\rr d)$
and \eqref{DualFrames} are dual frames, when $\Lambda =\ep \zz {d}$,
$\ep \in (0,\ep _0]$ and $\psi = (S_{\phi ,\phi}^\Lambda )^{-1}\phi$.
\end{enumerate}
\end{prop}

\par

We also recall the following restatement of \cite[Theorem 3.7]{Toft12} (see also
Corollaries 12.2.5 and 12.2.6 in \cite{Gc2} and Theorem 3.7 in \cite{GaSa}). Here
and in what follows we let $\Lambda ^2=\Lambda \times \Lambda$.

\par

\begin{prop}\label{ConseqThmS}
Let $\Lambda$ be as in \eqref{LambdaDef},
$\mabfp  \in
(0,\infty ]^{2d}$, $r\in (0,1]$ be such that $r\le \min (\mabfp )$, $\sigma
\in \operatorname {S}_{2d}$, and let $\omega ,v\in \mascP _E(\rr {2d})$
be such that $\omega$ is $v$-moderate. Also let
$\phi ,\psi \in M^r_{(v)}(\rr d)$ be such that \eqref{DualFrames}
are dual frames to each other. Then the following is true:
\begin{enumerate}
\item The operators $S_{\phi ,\psi} ^\Lambda \equiv D_\psi \circ C_\phi$ and
$S_{\psi ,\phi} ^\Lambda \equiv D_\phi \circ C_\psi$ are both the identity map
on $M^\mabfp _{\sigma ,(\omega )}(\rr d)$, and if  $f\in M^\mabfp
_{(\omega )}(\rr d)$, then
\begin{align}
f &= \sum _{j,\iota \in \Lambda} (V_\phi f)(j,\iota )
e^{i\scal {\cdo }{\iota }}\psi (\cdo -j)\notag
\\[1ex]
&=
\sum _{j,\iota \in \Lambda } (V_\psi f)(j,\iota )
e^{i\scal {\cdo }{\iota }}\phi (\cdo -j),\label{GabExpForm}
\end{align}
with unconditional norm-convergence in $M^\mabfp _{\sigma ,(\omega )}$
when $\max (\mabfp ) <\infty$, and with convergence in
$M^\infty _{(\omega)}$ with respect to the weak$^*$ topology otherwise;

\vrum

\item if  $f\in M^\infty _{(1/v)}(\rr d)$, then
$$
\displaystyle{\nm f{M^\mabfp _{\sigma ,(\omega )}}
\asymp
\nm {V_\phi f}
{\ell ^\mabfp _{\sigma ,(\omega )} (\Lambda ^2) }
\asymp
\nm {V_\psi f}
{\ell ^\mabfp _{\sigma ,(\omega )} (\Lambda ^2)} }.
$$
\end{enumerate}
\end{prop}

\par

Let $v$, $\phi$ and $\Lambda$ be as in Proposition \ref{ThmS}. Then
$(S_{\phi ,\phi}^\Lambda )^{-1}\phi$
is called the \emph{canonical dual window of  $\phi$}, with respect to
$\Lambda$. We have
$$
S_{\phi ,\phi}^\Lambda (e^{i\scal \cdo {\iota }}f(\cdo -j)) =
e^{i\scal \cdo {\iota }}(S_{\phi ,\phi}^\Lambda f)(\cdo -j),
$$
when $f\in M^\infty _{(1/v)}(\rr d)$ and $j,\iota \in \Lambda$.
The series in \eqref{GabExpForm}
are called \emph{Gabor expansions of $f$} with respect to $\phi$ and
$\psi$.

\par

\begin{rem}\label{RemThmS}
There are several ways to achieve dual frames
\eqref{DualFrames} satisfying the required properties in
Proposition \ref{ConseqThmS}. In fact, let
$v,v_{0}\in \mascP _E(\rr {2d})$ be submultiplicative such that
$\omega $ is $v$-moderate and $L^1_{(v_0)}(\rr {2d})\subseteq
L^r(\rr {2d})$. Then Proposition \ref{ThmS} guarantees that
for some choice of $\phi ,\psi \in M^1_{(v_0v)}(\rr d)\subseteq
M^r_{(v)}(\rr d)$ and lattice $\Lambda$ in \eqref{LambdaDef},
the sets in
\eqref{DualFrames} are dual frames to each others, and that
$\psi = (S_{\phi ,\phi}^\Lambda )^{-1}\phi$.
\end{rem}

\par

In the sequel we usually assume that $\Lambda =\Lambda _\ep$,
with $\ep >0$ small enough
such that the hypotheses in Propositions \ref{ThmS} and \ref{ConseqThmS}
are fulfilled, and that the window functions and their
duals belong to $M^r_{(v)}$ for every $r>0$. This
is always possible, in view of Remark \ref{RemThmS}.

\subsection{Classes of matrices}\label{subsec1.5}

\par

In what follows we let
$\Lambda$ be a in \eqref{LambdaDef}, $A$ be the complex matrix
$(a(j,k))_{j,k\in \Lambda}$, $p,q\in
(0,\infty ]$, $\omega$ be a map from $\Lambda ^2$
to $\mathbf R_+$, and
\begin{multline}\label{haomegadef}
h_{A,p,\omega }(k) \equiv \nm {H_{A,\omega}(\cdo ,k)}{\ell ^p},
\\[1ex]
\text{where}
\quad
H_{A,\omega}(j,k)=a(j,j-k)\omega (j,j-k).
\end{multline}

\par

\begin{defn}\label{matrixset1}
Let $0<p,q\le \infty$, $\Lambda$ be as in \eqref{LambdaDef}
and let $\omega$ be a map
from $\Lambda ^2$ to $\mathbf R_+$.
\begin{enumerate}
\item The set $\mathbb U_0(\Lambda )$ consists of matrices $(a(j,k))_{j,k\in \Lambda}$
such that at most finite numbers of $a(j,k)$ are non-zero;

\vrum

\item The set $\mathbb U^{p,q}(\omega ,\Lambda )$ consists of all
matrices $A=(a(j,k))_{j,k\in \Lambda}$ such that
$$
\nm A{\mathbb U^{p,q}(\omega ,\Lambda )} \equiv \nm {h_{A,p,\omega }}
{\ell ^q(\Lambda )},
$$
is finite, where $h_{A,p,\omega }$ is given by \eqref{haomegadef}. Furthermore,
$\mathbb U^{p,q}_0(\omega ,\Lambda )$ is the completion of
$\mathbb U_0(\Lambda )$ under the quasi-norm $\nm {\cdo}{\mathbb
U^{p,q}(\omega ,\Lambda )}$.
\end{enumerate}
\end{defn}

\par

For conveniency we set $\mathbb U^p(\omega ,\Lambda )=\mathbb
U^{p,p}(\omega ,\Lambda )$, and if $\omega =1$ everywhere, then
we set $\mathbb U^{p,q}(\Lambda )= \mathbb U^{p,q}(\omega ,
\Lambda )$ and $\mathbb U^{p}(\Lambda )=
\mathbb U^{p}(\omega ,\Lambda )$.

\par

\subsection{Pseudo-differential operators}

\par

Next we recall some properties in pseudo-differential calculus.
Let $s\ge 1/2$, $a\in \maclS _s 
(\rr {2d})$, and $t\in \mathbf R$ be fixed. Then the
pseudo-differential operator $\op _t(a)$
is the linear and continuous operator on $\maclS _s (\rr d)$, given by
\begin{equation}\label{e0.5}
(\op _t(a)f)(x)
=
(2\pi  ) ^{-d}\iint a((1-t)x+ty,\xi )f(y)e^{i\scal {x-y}\xi }\,
dyd\xi .
\end{equation}
For general $a\in \maclS _s'(\rr {2d})$, the
pseudo-differential operator $\op _t(a)$ is defined as the continuous
operator from $\maclS _s(\rr d)$ to $\maclS _s'(\rr d)$ with
distribution kernel
\begin{equation}\label{atkernel}
K_{a,t}(x,y)=(2\pi )^{-d/2}(\mascF _2^{-1}a)((1-t)x+ty,x-y).
\end{equation}
Here $\mascF _2F$ is the partial Fourier transform of $F(x,y)\in
\maclS _s'(\rr {2d})$ with respect to the $y$ variable. This
definition makes sense, since the mappings
\begin{equation}\label{homeoF2tmap}
\mascF _2\quad \text{and}\quad F(x,y)\mapsto F((1-t)x+ty,y-x)
\end{equation}
are homeomorphisms on $\maclS _s'(\rr {2d})$.
In particular, the map $a\mapsto K_{a,t}$ is a homeomorphism on
$\maclS _s'(\rr {2d})$.

\par

The standard (Kohn-Nirenberg) representation, $a(x,D)=\op (a)$, and
the Weyl quantization $\op ^w(a)$ of $a$ are obtained by choosing
$t=0$ and $t=1/2$, respectively, in \eqref{e0.5} and \eqref{atkernel}.

\par

\begin{rem}\label{BijKernelsOps}
By Fourier's inversion formula, \eqref{atkernel} and the kernel theorem
\cite[Theorem 2.2]{LozPerTask} for operators from
Gelfand-Shilov spaces to their duals,
it follows that the map $a\mapsto \op _t(a)$ is bijective from $\maclS _s'(\rr {2d})$
to the set of all linear and continuous operators from $\maclS _s(\rr d)$
to $\maclS _s'(\rr {2d})$.
\end{rem}

\par

By Remark \ref{BijKernelsOps}, it follows that for every $a_1\in \maclS _s '(\rr {2d})$
and $t_1,t_2\in \mathbf R$, there is a unique $a_2\in \maclS _s '(\rr {2d})$ such that
$\op _{t_1}(a_1) = \op _{t_2} (a_2)$. By Section 18.5 in \cite{Ho1},
the relation between $a_1$ and $a_2$
is given by
\begin{equation}
\label{calculitransform}
\op _{t_1}(a_1) = \op _{t_2}(a_2)
\quad \Longleftrightarrow \quad
a_2(x,\xi )=e^{i(t_1-t_2)\scal {D_x }{D_\xi}}a_1(x,\xi ).
\end{equation}

\par

We also recall that $\op _t(a)$ is a rank-one operator, i.{\,}e.
\begin{equation}\label{trankone}
\op _t(a)f=(2\pi )^{-d/2}(f,f_2)f_1, \qquad f\in \maclS _s(\rr d),
\end{equation}
for some $f_1,f_2\in \maclS _s '(\rr d)$,
if and only if $a$ is equal to the \emph{$t$-Wigner distribution}
\begin{equation}\label{wignertdef}
W_{f_1,f_2}^{t}(x,\xi ) \equiv \mascF (f_1(x+t\cdo
)\overline{f_2(x-(1-t)\cdo )} )(\xi ),
\end{equation}
of $f_1$ and $f_2$. If in addition $f_1,f_2\in L^2(\rr d)$, then $W_{f_1,f_2}^{t}$
takes the form
\begin{equation}\label{wignertdef2}
W_{f_1,f_2}^{t}(x,\xi ) = (2\pi )^{-d/2}\int _{\rr d} f_1(x+ty)
\overline{f_2(x-(1-t)y)}e^{-\scal y\xi} \, dy.
\end{equation}
(Cf. \cite{BoDoOl1}.) Since the Weyl case is of peculiar interests,
we also set $W_{f_1,f_2}=W_{f_1,f_2}^{t}$,
when $t=1/2$.

\par

\subsection{Schatten-von Neumann classes}

\par

Let $\maclB (V_1,V_2)$ denote the set of all linear and continuous
operators from the quasi-normed space $V_1$ to the quasi-normed
space $V_2$, and let $\nm \cdo{\maclB (V_1,V_2)}$ denote corresponding
quasi-norm. Let $\mascH _k$, $k=1,2,3$, be Hilbert spaces and
$T\in \maclB (\mascH _1, \mascH _2)$. Then the
\emph{singular value} of $T$ of order $j\ge 1$ is defined by
$$
\sigma _j(T) = \sigma _j(T ,\mascH _1, \mascH _2)
\equiv \inf \nm {T-T_0}{\maclB (\mascH _1, \mascH _2))},
$$
where the infimum is taken over all linear operators $T_0$ from
$\mascH _1$ to $\mascH _2$ of rank at most $j-1$.
The set $\mascI _p(\mascH _1, \mascH _2)$
of Schatten-von Neumann operators from
$\mascH _1$ to $\mascH _2$
of order $p\in (0,\infty ]$ is the set of all $T\in \maclB
(\mascH _1, \mascH _2)$ such that
\begin{equation}\label{SchattenNormBanach}
\nm T{\mascI _p(\mascH _1, \mascH _2)}
\equiv \nm {\{Ê\sigma _j(T)Ê\}Ê_{j\ge 1}}{\ell ^p}
\end{equation}
is finite.  We observe that $\mascI _p(\mascH _1, \mascH _2)$ is contained in
the set of compact operators from $\mascH _1$
to $\mascH _2$, when $p<\infty$.

\par

\par

We recall that if $p_0,p_1,p_2\in (0,\infty ]$, then
\begin{equation}\label{YoungSchatten}
\begin{gathered}
\nm {T_2\circ T_1}{\mascI _{p_0}(\mascH _1,\mascH _3)}
\le
\nm {T_1}{\mascI _{p_1}(\mascH _1,\mascH _2)}
\nm {T_2}{\mascI _{p_2}(\mascH _2,\mascH _3)}
\quad \text{when}
\\[1ex]
T_1\in \mascI _{p_1}(\mascH _1,\mascH _2),
\quad T_2\in \mascI _{p_2}(\mascH _2,\mascH _3),\quad 
\frac 1{p_1}+\frac 1{p_2}= \frac 1{p_0},
\end{gathered}
\end{equation}
and refer to \cite{Si,BS} for
more facts about Schatten-von Neumann classes.

\par

For convenience we set
$$
\mascI _p(\omega _1,\omega _2) \equiv
\mascI _p(\mascH _1, \mascH _2),
$$
when $\mascH _k=M^2_{(\omega _k)}(\rr d)$, for some $\omega _k\in
\mascP _E(\rr {2d})$, $k=1,2$. Moreover, if $t\in \mathbf R$ and
then $s_{t,p}(\omega _1,\omega _2)$ is the set of all $a\in \maclS _{1/2}'(\rr {2d})$
such that $\op _{t}(a)\in \mascI _p(\omega _1,\omega _2)$, and we set
$$
\nm a{s_{t,p}(\omega _1,\omega _2)}\equiv \nm {\op _{t}(a)}{\mascI
_p(\omega _1,\omega _2)}.
$$
We also set $s_p^w(\omega _1,\omega _2)=s_{t,p}(\omega _1,\omega _2)$
in the Weyl case, i.{\,}e. when $t=1/2$.
Moreover, if $\omega _1=\omega _2=1$, then we set
$s_{t,p}(\rr {2d})=s_{t,p}(\omega _1,\omega _2)$ and
$s_p^w(\rr {2d})=s_p^w(\omega _1,\omega _2)$.

\par

We recall that
$s_{t,p}(\omega _1,\omega _2)$
is a quasi-Banach space under the quasi-norm $a\mapsto
\nm a{s_{t,p}(\omega _1,\omega _2)}\equiv \nm {\op _t(a)}{\mascI _{p}
(\omega _1,\omega _2)}$. Furthermore, if in addition $p\ge 1$, then
$s_{t,p}(\omega _1,\omega _2)$ is a Banach space.

\par

By Remark \ref{BijKernelsOps} it follows that the map $a\mapsto \op _t(a)$
from $s_{t,p}(\omega _1,\omega _2)$ to $\mascI _p(\omega _1, \omega _2)$
is bijective and norm preserving.

\par

\subsection{Symplectic vector spaces and H{\"o}rmander symbol
classes}\label{subsec1.8}

\par

A real vector space $W$
of dimension $2d$ is called \emph{symplectic} if there is a non-degenerate
and anti-symmetric bilinear form $\sigma$ (the symplectic form). By choosing symplectic
coordinates $e_1,\dots ,e_d,\ep _1,\dots ,\ep _d$ in $W$, it follows that
$$
\sigma (X,Y) = \scal y\xi -\scal x\eta ,
$$
with
$$
X=(x,\xi )= \sum _{k=1}^d(x_je_j+\xi _j\ep _j)\in W,
\quad
Y=(y,\eta )= \sum _{k=1}^d(y_je_j+\eta _j\ep _j)\in W,
$$
which allows us to identify $W$ with the phase space $T^*V$ for some vector
space $V$ of dimension $d$, or by $T^*\rr d\simeq \rr {2d}$.

\par

The symplectic Fourier transform $\mascF _\sigma$ is the linear and
continuous map on $\maclS _{1/2}'(W)$, given by
$$
(\mascF _\sigma a)(X) = \pi ^{-d}\int _W a(Y)e^{2i\sigma (X,Y)}\, dY
$$
when $a\in \mascS (W)$. If $X=(x,\xi )\in T^*\rr d=W$, then it follows that
$(\mascF _\sigma a)(X) = 2^d\widehat a(-2\xi ,2x)$.

\par

Next we recall some notions on H{\"o}rmander symbol classes, $S(m,g)$,
parameterized by the Riemannian metric $g$ and the weight function $m$
on the $2d$ dimensional symplectic vector space $W$ (see e.{\,}g.
\cite{BoCh,BuTo,Ho0,Ho1,Le,Toft4}).
The reader who is not interested of the Schatten-von Neumann
results in Section \ref{sec4} of pseudo-differential operators
with symbols in $S(m,g)$ may pass to the next section.

\par

The H{\"o}rmander class $S(m,g)$ consists of all $a\in C^\infty (W)$ such that
$$
\nm a{m,N}^g \equiv \sum _{k=0}^N \sup _{X\in W}(|a|_k^g(X)/m(X)),
\quad \text{where}\quad
|a|_k^g(X) = \sup |a^{(k)}(X;Y_1,\dots ,Y_k)|.
$$
Here the latter supremum is taken over all $Y_1,\dots ,Y_k\in W$ such that
$g_X(Y_j)\le 1$, $j=1,\dots ,k$, and $|a|_0^g(X)$ is interpreted as $|a(X)|$.

\par

We need to add some conditions on $m$ and $g$. The metric $g$ is called
\emph{slowly varying} if there are positive constants $c$ and $C$ such that
\begin{equation}\label{Eq:SlowlyVar}
C^{-1}g_X \le g_Y \le Cg_X,\quad \text{when}\quad  X,Y\in W
\end{equation}
satisfy $g_X(X-Y)\le c$, and $m$ is called $g$-continuous when \eqref{Eq:SlowlyVar}
holds with $m(X)$ and $m(Y)$ in place of $g_X$ and $g_Y$,
respectively, provided $g_X(X-Y)\le c$.

\par

For the Riemannian metric $g$ on $W$, the \emph{dual} metric $g^\sigma$ with respect to the
symplectic form $\sigma$, and the \emph{Planck's function} $h_g$ are defined by
$$
g_X^\sigma (Z) \equiv \sup _{g_X(Y)\le 1}\sigma (Y,Z)^2
\qquad \text{and}\qquad
h_g(X) \equiv \sup _{g_X^\sigma (Y)\le 1}g_X(Y)^{1/2}.
$$
Moreover, if $g$ is slowly varying and $m$ is $g$-continuous, then $g$ is called
$\sigma$-temperate if there are positive constants $C$ and $N$ such that
\begin{equation}\label{Eq:Sigmatemp}
g_Y(Z)\le Cg_X(Z) (1+g_Y(X-Y))^N,\qquad X,Y,Z\in W,
\end{equation}
and $m$ is called $(\sigma ,g)$-temperate if it is $g$-continuous and
\eqref{Eq:Sigmatemp} holds with $m(X)$ and $m(Y)$ in place of
$g_X(Z)$ and $g_Y(Z)$, respectively.

\par

\begin{defn}
Let $g$ be a Riemannian metric on W. Then $g$ is called \emph{feasible} if it is slowly
varying and $h_g\le 1$ everywhere. Furthermore, $g$ is called \emph{strongly feasible}
if it is feasible and $\sigma$-temperate.
\end{defn}

\par

We remark that the H{\"o}rmander class $S^r_{\rho ,\delta}$ in \cite{Ho1}, the $\SG$-class
in \cite{Co,Pa}, the Shubin classes in \cite[Definition 23.1]{Sh} and other well-known
families of symbol classes are given by $S(m.g)$ for suitable choices of
strongly feasible metrics $g$ and $(\sigma ,g)$-temperate weights $m$.

\par

\section{Estimates for matrices}\label{sec2}

\par

In this section we deduce continuity and Schatten-properties for matrices
in the classes $\mathbb U^{p ,q}(\omega ,\Lambda )$.
In the first part we achieve convenient factorization
results for matrices in the case when $p=q$ (cf. Theorem \ref{factorizationprop}).
Thereafter we establish the continuity properties (cf. Theorem \ref{matrixcont2}).
In the last part of the section we combine these factorizations and continuity
results to establish Schatten properties for matrix operators
(cf. Theorem \ref{MatrixSchatten}).

\par

Theorem \ref{factorizationprop} below allows factorizations of matrices in
$\mathbb U^{p}(\omega ,\Lambda )$ in suitable ways, when $\Lambda$
is given by \eqref{LambdaDef}. Here the involved weights
should fulfill
\begin{alignat}{2}
\omega _1(j,j)\omega _2(j,k)&\le \omega _0(j,k),&
\qquad j,k &\in \Lambda \label{weightcond1}
\intertext{or}
\omega _1(j,k)\omega _2(k,k)&\le \omega _0(j,k),&
\qquad j,k &\in \Lambda \label{weightcond2}
\end{alignat}
and the involved Lebesgue exponents should satisfy the H{\"o}lder
condition
\begin{equation}\label{holdercond}
\frac 1{p_0} \le \frac 1{p_1} + \frac 1{p_2}, 
\end{equation}


\par

\begin{thm}\label{factorizationprop}
Let $\Lambda$ be as in \eqref{LambdaDef}, $p _l\in (0, \infty ]$
be such that \eqref{holdercond} hold, $\omega _l$, $l=0,1,2$, be
weights on $\rr {2d}$ , and let $A_0\in
\mathbb U^{p_0}(\omega _0,\Lambda )$. Then the following is true:
\begin{enumerate}
\item if \eqref{weightcond1} holds, then
$A_0=A_1\cdot A_2$ for some $A_l\in \mathbb U^{p_l}(\omega _l,\Lambda )$,
$l=1,2$. Furthermore,  $A_1$ can be chosen as a diagonal matrix;

\par

\item if \eqref{weightcond2} holds, then
$A_0=A_1\cdot A_2$ for some $A_l\in \mathbb U^{p_l}(\omega _l,\Lambda )$,
$l=1,2$. Furthermore,  $A_2$ can be chosen as a diagonal matrix.
\end{enumerate}

\par

Moreover, the matrices  in {\rm{(1)}} and {\rm{(2)}} can be chosen
such that
\begin{equation}\label{multcont}
\nm {A_1}{\mathbb U^{p_1}(\omega _1,\Lambda )}
\nm {A_2}{\mathbb U^{p_2}(\omega _2,\Lambda )} \le \nm
{A_0}{\mathbb U^{p_0}(\omega _0,\Lambda )}.
\end{equation}
\end{thm}

\par

\begin{proof}
It is no restrictions to assume that equality is attained in \eqref{holdercond},
and by transposition it also suffices to prove (1).

\par

We only prove the result for $p_0<\infty$. The small modifications to the
case when $p_0=\infty$ are left for the reader. Let $a(j,k)$ be the matrix
elements for $A_0$, and let $A_1 =(b(j,k))$ and $A_2=(c(j,k))$
be the matrices such that
$$
b(j,k) =
\begin{cases}
\big (\omega _1(j,j)\big )^{-1}
\displaystyle{\left (\sum _{m}|a(j,m)\omega _0(j,m)|^{p_0}\right
)^{1/{p_1}}},&\quad j=k
\\[3ex]
0,&\quad   j\neq k
\end{cases}
$$
and $c(j,k)=a(j,k)/b(j,j)$ when $b(j,j)\neq 0$, and $c(j,k)=0$ otherwise.

\par

Since
$$
b(j,j)\ge (\omega _1(j,j))^{-1}|a(j,k)\omega _0(j,k)|^{p_0/p_1},
\quad \text{and}\quad
\frac 1{p_0} - \frac 1{p_1} = \frac 1{p_2},
$$
\eqref{weightcond1} gives
\begin{multline*}
|c(j,k)\omega _2(j,k)|\le |a(j,k)|^{p_0/p_2}\omega _1(j,j)\omega _2(j,k)
/\omega _0(j,k)^{p_0/p_1}
\\[1ex]
\le |a(j,k)|^{p_0/p_2}\omega _0(j,k)^{p_0/p_2}.
\end{multline*}
This in turn gives
\begin{multline*}
\nm {A_1}{\mathbb U^{p_1}(\omega _1,\Lambda )} = \left ( \sum _{j,k}|b(j,k)
\omega _1(j,k)|^{p_1}\right )^{1/{p_1}}
\\[1ex]
=
\left (   \left (  \sum _j  \left (   \sum _{m}   |a(j,m)\omega _0(j,m)|^{p_0}
\right )^{1/p_1}  \right )^{p_1}  \right )^{1/p_1} =
\nm {A_0}{\mathbb U^{p_0}(\omega _0,\Lambda )}^{p_0/p_1},
\end{multline*}
and
\begin{multline*}
\nm {A_2}{\mathbb U^{p_2}(\omega _2,\Lambda )} =
\left (\sum _{j,k} |c(j,k)\omega _2(j,k)| ^{p_2}\right )^{1/{p_2}}
\\[1ex]
\le 
\left ( \sum _{j,k} |a(j,k)\omega _0(j,k)|^{p_0}\right )^{1/{p_2}} =
\nm {A_0}{\mathbb U^{p_0}(\omega _0,\Lambda )}^{p_0/p_2}.
\end{multline*}
Hence $A_l\in \mathbb U^{p_l}(\omega _l,\Lambda )$, $l=1,2$. Since $A_0=A_1\cdot
A_2$ and $p_0/p_1+p_0/p_2=1$, the result follows.
\end{proof}

\par

If the weights $\omega _l$, $l=0,1,2$, fulfill
\begin{equation}\label{weightcond3}
\omega _1(j,m)\omega _2(m,k)\le \omega _0(j,k),\qquad
\text{for every}\ j,k,m\in \Lambda ,
\end{equation}
then it is evident that both \eqref{weightcond1} and \eqref{weightcond2}
are fulfilled. Hence the following result is a special case of Theorem
\ref{factorizationprop}.

\par

\begin{prop}\label{factorizationprop2}
Let $\Lambda$ be as in \eqref{LambdaDef}, $p _l\in (0, \infty ]$
and let $\omega _l$,
$l=0,1,2$, be weights on $\rr {2d}$ such that \eqref{holdercond} and
\eqref{weightcond3} hold,
and let $A_0\in \mathbb U^{p_0}(\omega _0,\Lambda )$. Then
$A_0=A_1\cdot A_2$ for some $A_l\in \mathbb U^{p_l}(\omega _l,\Lambda )$,
$l=1,2$. Moreover, the matrices $A_1$ and $A_2$
can be chosen such that
\eqref{multcont} holds.
\end{prop}

\medspace

Next we deduce continuity results for matrix operators. 
We recall that if $A=(a(j,k))_{j,k\in \Lambda}$ is a matrix, then $Af$ is uniquely
defined as an element in $\ell (\Lambda )$ when $f\in \ell _0(\Lambda )$, i.{\,}e.
\begin{alignat}{2}
A\, &:\, &  \ell _0(\Lambda ) &\mapsto \ell (\Lambda ). \label{Abmap1}
\intertext{Furthermore, if in addition $A$ belongs to $\mathbb U_0(\Lambda )$,
then $Af$ is uniquely defined as an element in $\ell _0(\Lambda )$ when $f\in l(\Lambda )$,
i.{\,}e.}
A\, &:\, & \ell (\Lambda ) &\mapsto \ell _0(\Lambda )\phantom ,\qquad \text{when}\quad
A\in \mathbb U_0(\Lambda ). \label{Abmap2}
\end{alignat}

\par

For $p\in [1,\infty ]$, its conjugate exponent $p'\in [1,\infty ]$ is usually
defined by $1/p+1/p'=1$. For $p$ belonging to the larger interval $(0,\infty ]$
it is convenient to extend the definition of $p'$ as
$$
p'=
\begin{cases}
\ \ 1, & p=\infty
\\[1ex]
\displaystyle{\frac p{p-1}}, & 1<p<\infty
\\[1ex]
\ \ \infty , & 0<p\le 1 .
\end{cases}
$$

\par

The next theorem is the main result concerning the continuity
for matrix operators.

\par

\begin{thm}\label{matrixcont2}
Let $\sigma \in S_d$, $\theta \in \mathbf R^d_*$, $\Lambda =T_\theta \zz d$,
$\omega _l$ be weights on $\Lambda$,
$l=1,2$, and $\omega _0$ be a weight on $\Lambda \times \Lambda$ such that
\eqref{weightineq1} holds. Also let $\mabfp _1,\mabfp _2\in (0,\infty]^n$, and
$p,q\in (0,\infty]$
be such that
\begin{equation}\label{pqconditions}
\frac 1{\mabfp _2}-\frac 1{\mabfp _1} = \frac 1{p}+\min \left ( 0,\frac 1{q}-1\right ) ,
\quad 
q \le \min (\mabfp _2)\le
\max (\mabfp _2) \le p,
\end{equation}
and let $A\in \mathbb U^{p,q}(\omega _0,\Lambda )$.
Then $A$ from $\ell _0(\Lambda )$ to $\ell (\Lambda )$ is uniquely extendable to a continuous
map from $\ell ^{\mabfp _1} _{\sigma ,(\omega _1)}(\Lambda )$ to $\ell ^{\mabfp _2}
_{\sigma ,(\omega _2)}(\Lambda )$, and
\begin{equation}\label{Anormest}
\nm A{\maclB (\ell ^{\mabfp _1}_{\sigma ,(\omega _1)} (\Lambda )
, \ell ^{\mabfp _2}_{\sigma ,(\omega _2)}(\Lambda ))}\le
\nm A{\mathbb U^{p,q}(\omega _0,\Lambda )}.
\end{equation}
\end{thm} 

\par

We note that \eqref{Anormest} is the same as
\begin{equation}\label{Anormest2}
\nm {Af}{\ell ^{\mabfp _2}_{\sigma ,(\omega _2)}(\Lambda )}\le \nm
A{\mathbb U^{p,q}(\omega _0,\Lambda )}
\nm f{\ell ^{\mabfp _1}_{\sigma ,(\omega _1)} (\Lambda )},\quad
f\in \ell ^{\mabfp _1}_{\sigma ,(\omega _1)} (\Lambda ).
\end{equation}

\par

\begin{proof}
By permutation of the Lebesgue exponents, we reduce ourself to the case when
$\sigma$ is the identity map. We consider the cases $q\le 1$ and $q\ge 1$ separately.
Let $f\in \ell ^{\mabfp _1}_{(\omega _1)}$,
$h=h_{A,\infty ,\omega}$ be the same as in Definition \ref{matrixset1},
with $\omega =\omega _0$, and set
\begin{equation*}
c(k) = |f(k)\omega _1(k)|,
\quad
a_0 (j,k) = |a(j,j-k)\omega _0(j,j-k)|
\quad \text{and}\quad
g = Af.
\end{equation*}

\par

First we consider the case when $p=\infty$, $q\le 1$, and in addition
$A\in \mathbb U_0(\Lambda )$.
Then $\mabfp _1=\mabfp _2$, and we get
\begin{multline}\label{MatrixComp1}
|g(j)\omega _2(j)|\le \sum _k  |a(j,k)\omega _0(j,k)| \, c(k) 
\\[1ex]
= \sum _k  a_0(j,k) \, c(j-k) 
\\[1ex]
\le \sum _k h(k) c(j-k) =(h*c)(j).
\end{multline}
Hence, Corollary 2.2 in \cite{Toft12} gives
\begin{equation}\label{MatrixComp2}
\nm {Af}{\ell ^{\mabfp _2}_{(\omega _2)}} = \nm {g\cdot \omega _2}
{\ell ^{\mabfp _2}}\le \nm h{\ell ^q}\nm c{\ell ^{\mabfp _1}}
= \nm A{\mathbb U^{\infty ,q}(\omega _0,\Lambda )}
\nm f{\ell ^{\mabfp _1}_{(\omega _1)}},
\end{equation}
and the result follows in this case.

\par

For general $A\in \mathbb U^{\infty ,q}(\omega _0,\Lambda )$
we decompose $A$ and $f$ into
$$
A=A_1-A_2 +i(A_3-A_4)
\quad \text{and}\quad 
f=f_1-f_2 +i(f_3-f_4),
$$
where $A_j$ and $f_k$ only have non-negative entries, chosen
as small as possible. By Beppo Levi's theorem and the estimates above
it follows that $A_jf_k$ is uniquely defined as an element in
$\ell ^{\mabfp _1}_{(\omega _1)}$. It also follows from these estimates
\eqref{Anormest} holds, and we have proved the result in the case $p=\infty$
and $q\le 1$.

\par

The case when $q\le 1$, $p<\infty$ and $A\in \mathbb U_0(\Lambda )$
is obtained by induction. Let
\begin{alignat*}{2}
G_0(j) &\equiv |g(j)\omega _2(j)|, &
\quad
b_0(j,k)&\equiv a_0(j,k) = |a(j,j-k)\omega _0(j,j-k)|,
\\[1ex]
c_0(j) &\equiv c(j), & \quad
\Lambda _0 &\equiv \{ 0\}
\quad \text{and}\quad
\Lambda _0^* \equiv \Lambda = \theta _1\mathbf Z\times \cdots \times \theta _d\mathbf Z. 
\end{alignat*}
Also let
\begin{alignat*}{2}
\mabfp _{l,m} &\equiv  (p_{l,1},\dots ,p_{l,m}) &
\quad \text{when}\quad
\mabfp _{l} &= (p_{l,1},\dots ,p_{l,d}),\ l=1,2,
\\[1ex]
\Lambda _m &\equiv \theta _1\mathbf Z \times \cdots \times \theta _m\mathbf Z, &
\quad
\Lambda _m^* &\equiv \theta _{m+1}\mathbf Z \times \cdots \times \theta _d\mathbf Z,
\\[1ex]
\mabfj _m &= (j_{m+1},\dots ,j_d)\in \Lambda _m^* &
\quad \text{and}\quad
\mabfk _m &= (k_{m+1},\dots ,k_d)\in \Lambda _m^*
\intertext{when}
j&=(j_1,\dots ,j_d)\in \Lambda &
\quad \text{and}\quad
k&=(k_1,\dots ,k_d)\in \Lambda ,
\end{alignat*}
and let
$$
b_m(\mabfj _m,\mabfk _m) \equiv \nm {a_{0,m}(\mabfj _m,\cdo ,\mabfk _m)}
{\ell ^q)(\Lambda _m)}, \quad
m =1,\dots ,d,
$$
where
$$
a_{0,m}(\mabfj _m,k) \equiv \nm {a_0(\cdo ,\mabfj _m,k)}{\ell ^p(\Lambda _m)}.
$$
Define inductively
\begin{gather*}
G_m(\mabfj _m) \equiv \nm {G_{m-1}(\cdo ,\mabfj _m)}{\ell ^{p_{2,m}} (\theta _m\mathbf Z)}
\quad \text{and}\quad
c_m(\mabfj _m) \equiv \nm {c_{m-1}(\cdo ,\mabfj _m)}{\ell ^{p_{1,m}} (\theta _m\mathbf Z)},
\end{gather*}
when $m=1,\dots ,d$, where $\Lambda _d^*$, $G_d$ and $c_d$ are interpreted as
$$
\{ 0\} ,\quad
\nm {G_{d-1}}{\ell ^{p_{2,d}}(\theta _d\mathbf Z)} = \nm {Af}{\ell ^{\mabfp _2}_{(\omega _2)}}
\quad \text{and} \quad
\nm {c_{d-1}}{\ell ^{p_{1,d}}(\theta _d\mathbf Z)} = \nm f{\ell ^{\mabfp _1}_{(\omega _1)}},
$$
respectively. We claim
\begin{equation}\label{Eq:GbcIneq}
G_m(\mabfj _m) \le
\left (
\sum _{\mabfk _m\in \Lambda _m^*}\big ( b_m(\mabfj _m,\mabfk _m)
c_m(\mabfj_m -\mabfk _m)\big )^q,
\right )^{1/q}
\end{equation}
for $m=0,\dots ,d$.

\par

In fact, the case $m=0$ follows from the equality in \eqref{MatrixComp1}
and the fact that $q\le 1$. Suppose \eqref{Eq:GbcIneq} is true for 
$m-1$ in place of $m$, and let $r=p_{2,m}/q$. Then $r\in [1,\infty )$, since
$p<\infty$ and $q\le p_{2,m}$. Hence, \eqref{pqconditions}, and H{\"o}lder's
and Minkowski's inequalities in combination with the inductive assumptions give
\begin{multline*}
G_m(\mabfj _m)^q
\le
\left (
\sum _{j_m\in \theta _m\mathbf Z}
\left (
\sum _{\mabfk _{m-1}\in \Lambda _{m-1}^*}(b_{m-1}(\mabfj _{m-1},\mabfk _{m-1})
c_{m-1}(\mabfj _{m-1}-\mabfk _{m-1})
)^q
\right )^r
\right )^{1/r}
\\[1ex]
\le
\sum _{\mabfk _{m-1}\in \Lambda _{m-1}^*}
\left (
\sum _{j_m\in \theta _m\mathbf Z} 
(b_{m-1}(\mabfj _{m-1},\mabfk _{m-1})
c_{m-1}(\mabfj _{m-1}-\mabfk _{m-1})
)^{p_{2,m}}
\right )^{1/r}
\\[1ex]
\le
\sum _{\mabfk _{m-1}\in \Lambda _{m-1}^*}
\nm{b_{m-1}(\cdo ,\mabfj _m,\mabfk _{m-1})}{\ell ^p(\theta _m\mathbf Z)}^q
\nm {c_{m-1}(\cdo ,\mabfj _{m}-\mabfk _m)}{\ell ^{p_{1,m}}(\theta _m\mathbf Z)}^q.
\end{multline*}

\par

We have $c_m(\mabfj _m)= \nm {c_{m-1}(\cdo ,\mabfj _m)}{\ell ^{p_{1,m}}(\theta _m\mathbf Z)}$,
and
$$
\sum _{k_m\in \theta _m\mathbf Z}\nm{b_{m-1}(\cdo ,\mabfj _m,k_m,\mabfk _m)}{\ell ^p(\theta _m\mathbf Z)}^q
\le
\nm{a_{0,m}(\mabfj _m\cdo ,\mabfk _m)}{\ell ^p(\theta _m\mathbf Z)}^q,
$$
by Minkowski's inequality, and a combination of the previous inequalities give \eqref{Eq:GbcIneq}. Hence,
by induction we have that \eqref{Eq:GbcIneq} holds for every $m=0,\dots ,d$, and by letting $m=d$ we obtain
\eqref{Anormest} when $A\in \mathbb U_0(\Lambda )$. The result now follows for general
$\mathbb U^{p,q}(\omega _0,\Lambda )$ when $p<\infty$ and $q\le 1$ by the fact that
$\mathbb U_0(\Lambda )$ is dense in $\mathbb U^{p,q}(\omega _0,\Lambda )$.

\par

Next we consider the case $q\in (1,\infty ]$, and assume first that $p=\infty$.
Then
$$
\frac 1{\mabfp _1}+\frac 1q = 1+\frac 1{\mabfp _2}.
$$
Hence, if $A\in \mathbb U^{\infty ,q}(\omega _0,\Lambda )$ and $f\in \ell _0(\Lambda)$,
then \eqref{MatrixComp1} and Young's inequality give
$$
\nm g{\ell ^{\mabfp _2}_{(\omega _2)}} \le \nm h{\ell ^q}\nm c{\ell ^{\mabfp _1}},
$$
and \eqref{Anormest} follows in this case as well. Since $\max (\mabfp _1)<\infty$ when
$q>1$, the result follows for general $f\in \ell ^{\mabfp _1}_{(\omega _1)}(\Lambda)$
from the fact that $\ell _0(\Lambda)$ is dense in $\ell ^{\mabfp _1}_{(\omega _1)}(\Lambda)$.

\par

For general $p$, the result now follows by multi-linear interpolation between the cases
$(p,q)=(1,1)$ and $(p,q)=\{ \infty \} \times [1,\infty ]$, using Theorems 4.4.1 and 5.6.3
in \cite{BeLo}. The proof is complete.
\end{proof}

\par

The following consequence of the previous result is particularly important.
 
\par

\begin{cor}\label{matrixcont1}
Let $\Lambda$ be as in \eqref{LambdaDef}, $p\in (0,\infty ]$, $\omega _l$,
$l=1,2$ be weights on $\rr d$ and $\omega _0$ be a weight on
$\rr {2d}$ such that
\begin{equation}\label{weightineq1}
\frac {\omega _2(j)}{\omega _1(k)} \le \omega _0(j,k),\qquad j,k\in \Lambda
\end{equation}
holds. Also let $A\in \mathbb U^p (\omega _0,\Lambda )$. Then $A$ in
\eqref{Abmap1} is uniquely extendable to a continuous map from
$\ell ^{p'} _{(\omega _1)}(\Lambda )$ to $\ell ^p_{(\omega _2)}(\Lambda )$, and
\begin{equation}\label{normest}
\nm A{\maclB (\ell ^{p'} _{(\omega _1)}(\Lambda ), \ell ^p_{(\omega _2)}(\Lambda ))}
\le \nm A{\mathbb U^p(\omega _0,\Lambda )}.
\end{equation}
\end{cor}

\par

\par

The next result deals with Schatten-von Neumann properties for matrix operators.


\par

\begin{thm}\label{MatrixSchatten}
Let $\Lambda$ be as in \eqref{LambdaDef}, $\omega _l$, $l=1,2$ be weights
on $\rr d$ and $\omega _0$ be a weight on
$\rr {2d}$ such that \eqref{weightineq1} holds. Also let $p\in (0,2]$,
and let $A\in \mathbb U^p (\omega _0,\Lambda )$. Then
$A\in \mascI _p(\ell ^2_{(\omega _1)}(\Lambda ),
\ell ^2_{(\omega _2)}(\Lambda ))$, and
\begin{equation}\label{Schattnormest1}
\nm A{\mascI _p(\ell ^2_{(\omega _1)}(\Lambda ),\ell ^2_{(\omega _2)}(\Lambda ))}
\le \nm A{\mathbb U^p(\omega _0,\Lambda )}.
\end{equation}
\end{thm}

\par

\begin{proof}
We may assume that equality is attained in \eqref{weightineq1}, and that
$\nm A{\mathbb U^p(\omega _0,\Lambda )}=1$. Then
it follows that
$$
\mascI _2(\ell ^2_{(\omega _1)}(\Lambda ),\ell ^2_{(\omega _2)}(\Lambda )) =
\mathbb U^2 (\omega _0,\Lambda ),
$$
with equality in norms.

\par

First assume that $p=2/N$ for some integer $N\ge 3$, and let
$A\in \mathbb U^{2/N} (\omega _0,\Lambda )$. Also let
$\vartheta _1(j,k)=\omega _2(j)$, $\vartheta _m(j,k)=1$, $j=2,\dots ,N-1$ and
$\vartheta _N(j,k)=\omega _1(k)$. By Theorem \ref{factorizationprop} we have
$$
A=A_1\circ \cdots \circ A_N
$$
for some $A_m \in \mathbb U^2 (\vartheta _m,\Lambda )$ which satisfy
$\nm {A_m}{\mathbb U^2 (\vartheta _m,\Lambda )}\le 1$, $m=1,\dots ,N$.

\par

By \eqref{YoungSchatten} we get
\begin{multline*}
\nm A{\mascI _{2/N}(\ell ^2_{(\omega _1)},\ell ^2_{(\omega _2)})} \le
\nm {A_1}{\mascI _{2}(\ell ^2,\ell ^2_{(\omega _2)})}
\nm {A_N}{\mascI _{2}(\ell ^2_{(\omega _1)},\ell ^2)}
\prod _{m=2}^{N-1} \nm {A_m}{\mascI _{2}(\ell ^2,\ell ^2)}
\\[1ex]
=
\prod _{m=2}^{N-1} \nm {A_m}{\mathbb U^2 (\vartheta _m,\Lambda )} \le 1,
\end{multline*}
and the result follows in the case $p=2/N$.

\par

The result is therefore true when $p=2/N$ for some integer $N\ge 3$,
and when $p=2$. For $p\in [2/N,2]$, the result now follows by (real) interpolation
between the cases $p=2$ and $p=2/N$, letting $q$, $p_\theta$, $p_k$, $q_k$ and
$\theta \in [0,1]$, $k=0,1$, in Teorema 3.2, (3.11) and (3.13) in \cite{BS} be chosen
such that
$$
q=p_\theta ,\quad q_0=p_0=\frac 2N ,\quad q_1=p_1=2
\quad \text{and}\quad
\frac 1{p_\theta} = \frac {1-\theta}{p_0}+\frac \theta{p_1}.
$$

\par

For general $p\in (0,2]$, the result now follows by choosing
$N\ge 3$ such that $p>2/N$. The proof is complete.
\end{proof}

\par

\par

\section{Continuity and Schatten-von Neumann
properties for pseudo-differential operators}\label{sec3}

\par

In this section we deduce continuity and Schatten-von Neumann
results for pseudo-differential operators with symbols in
modulation spaces. In particular we extend results in
\cite{GH1,Gc2,Toft2,Toft5,Toft11} to include Schatten and
Lebesgue parameters less than one.

\par

We start with the following result on continuity.

\par

\begin{thm}\label{thmOpCont}
Let $t\in \mathbf R$, $\sigma \in \operatorname{S}_{2d}$, $\omega _1,\omega _2
\in \mathscr P_{E}(\rr {2d})$ and $\omega _0\in \mathscr P_{E}(\rr {2d}\oplus \rr {2d})$
be such that
$$
\frac {\omega _2(x,\xi  )}{\omega _1
(y,\eta )} \lesssim \omega _0( (1-t)x+ty,t\xi +(1-t)\eta ,\xi -\eta ,y-x ).
$$
Also let $\mabfp _1,\mabfp _2\in (0,\infty]^{2d}$,
$p,q\in (0,\infty]$ be such that \eqref{pqconditions} hold, and let
$a\in M^{p,q}_{(\omega _0)}(\rr {2d})$. Then
$\op _t(a)$ from $\mathcal S_{1/2}(\rr d)$ to $\mathcal S_{1/2}'(\rr d)$
extends uniquely to a continuous map from $M^{\mabfp _1}
_{\sigma ,(\omega _1)}(\rr d)$ to
$M^{\mabfp _2}_{\sigma ,(\omega _2)}(\rr d)$, and
\begin{equation}\label{PsDOEst}
\nm {\op _t(a)}{\maclB (M^{\mabfp _1}_{\sigma ,(\omega _1)},
M^{\mabfp _2}_{\sigma ,(\omega _2)})}
\lesssim
\nm a{M^{p,q}_{(\omega _0)}}.
\end{equation}
\end{thm}

\par

We need some preparing lemmata for the proof. We recall that
$\Lambda ^2=\Lambda \times \Lambda$ when $\Lambda$ is a lattice.

\par

\begin{lemma}\label{aFrames}
Let $v\in \mascP _E(\rr {4d})$, $\phi _1,\phi _2\in \Sigma _1(\rr d)\setminus 0$,
and let
$$
\Phi (x,\xi )=\phi _1(x)\overline {\widehat \phi _2(\xi )}e^{-i\scal x\xi },
$$
Then there is a
lattice $\Lambda$ in \eqref{LambdaDef} such that
\begin{align*}
&\{ \Phi (x-j,\xi -\iota )e^{i(\scal x\kappa +\scal k\xi )} \} _{(j, \iota ),(k, \kappa )
\in \Lambda ^2}
\intertext{is a Gabor frame with canonical dual frame}
&\{ \Psi (x-j,\xi -\iota )e^{i(\scal x\kappa +\scal k\xi )} \}
_{(j,\iota ),(k,\kappa ) \in \Lambda ^2},
\end{align*}
where $\Psi = (S_{\Phi ,\Phi }^{\Lambda ^2\times \Lambda ^2})^{-1}\Phi$
belongs to $M^r _{(v)}(\rr {2d})$ for every $r>0$.
\end{lemma}

\par

Note that $\Phi$ in Lemma \ref{aFrames} is the Rihaczek (cross)-distribution
of $\phi _1$ and $\phi _2$ (cf. \cite{GrSt}).

\par

\begin{proof}
The result follows from Remark \ref{RemThmS}, and the fact that
$\Phi \in \Sigma _1(\rr {2d})\setminus 0$ in view of \cite[Theorem 3.1]{CaTo} or
\cite[Proposition 3.4]{CaWa} .
\end{proof}

\par

\begin{lemma}\label{PsDoDisc}
Let $\Lambda$, $\phi _1$, $\phi _2$, $\Phi$ and $\Psi$ be as in Lemma
\ref{aFrames}. Also let $v\in \mascP _E (\rr {4d})$, $a\in M^\infty _{(1/v)}(\rr {2d})$,
\begin{multline*}
c_0(\mabfj ,\mabfk) \equiv (V_\Psi a)(j,\kappa ,\iota -\kappa ,k-j)e^{i\scal {k-j}\kappa},
\\[1ex]
\text{where}\quad \mabfj =(j,\iota )\in \Lambda ^2 ,\ \mabfk = (k,\kappa )\in \Lambda ^2.
\end{multline*}
and let $A$ be the matrix $A=(c_0(\mabfj ,\mabfk )_{\mabfj ,\mabfk \in \Lambda ^2}$.
Then the following is true:
\begin{enumerate}
\item if $p,q\in (0,\infty ]$ and $\omega ,\omega _0\in \mascP _E(\rr {4d})$ satisfy
\begin{equation}\label{omega0omegaRel}
\omega (x,\xi ,y,\eta )\asymp \omega _0(x,\eta ,\xi -\eta ,y-x),
\end{equation}
then $a\in M^{p,q}_{(\omega _0)}(\rr {2d})$, if and only if
$A\in \mathbb U^{p,q}(\omega ,\Lambda ^2 )$, and then
$$
\nm a{M^{p,q}_{(\omega _0)}}\asymp
\nm A{\mathbb U^{p,q}(\omega ,\Lambda ^2 )}\text ;
$$

\vrum

\item $\op (a)$ as a map from $\maclS _{1/2}(\rr d)$ to $\maclS _{1/2}'(\rr d)$, given
by
\begin{equation}\label{OpaFactorization}
\op (a) = D_{\phi _1} \circ A \circ C_{\phi _2}.
\end{equation}
\end{enumerate}
\end{lemma}

\par

Some arguments in \cite{GrSt} appear in the proof of Lemma \ref{PsDoDisc}.

\par

\begin{proof}
We have
$$
|c_0(\mabfj ,\mabfj -\mabfk )| = | (V_\Psi a) (j,\iota -\kappa ,\kappa ,-k)|.
$$
Hence, Proposition \ref{ConseqThmS} (2) gives
$$
\nm A{\mathbb U^{p,q}(\omega ,\Lambda ^2 )}
=
\nm {V_\Psi a}{\ell ^{p,q}_{(\omega _0)}(\Lambda ^2 \times \Lambda ^2 )}
\asymp
\nm a{M^{p,q}_{(\omega _0)}},
$$
and (1) follows.

\par

Next we prove (2). Let $f\in \maclS _{1/2}(\rr d)$, and let
$$
c(\mabfj ,\mabfk ) = (V_\Psi a)(j,\iota ,\kappa ,k).
$$
By Proposition \ref{ConseqThmS} we have
$$
a= \sum _{\mabfj ,\mabfk \in \Lambda ^2 }c(\mabfj , \mabfk )\Phi _{\mabfj , \mabfk},
$$
where
$$
\Phi _{\mabfj , \mabfk}(x,\xi )= e^{i(\scal x\kappa +\scal k\xi)}\Phi (x-j,\xi -\iota ).
$$
This gives
$$
\op (a)  = \sum _{\mabfj ,\mabfk \in \Lambda ^2 }c(\mabfj , \mabfk )
\op (\Phi _{\mabfj , \mabfk}),
$$
and we shall evaluate $\op (\Phi _{\mabfj , \mabfk})f$.

\par

We have
$$
\Phi _{\mabfj , \mabfk}(x,\xi ) = \phi _1(x-j)\overline
{\widehat \phi _2(\xi -\iota )}e^{-i\scal {x-j}{\xi -\iota }}e^{i(\scal x\kappa +\scal k\xi )},
$$
and by straight-forward computations we get
$$
(\op (\Phi _{\mabfj , \mabfk})f)(x) = \phi _1(x-j)e^{i\scal x{\iota +\kappa }}e^{-i\scal j\iota}
F_0(\mabfj , \mabfk ),
$$
where
$$
F_0(\mabfj , \mabfk )
=
(2\pi )^{-d/2}\int \widehat f(\xi ) \overline {\widehat \phi _2(\xi -\iota )}
e^{i\scal {j+k}\xi}\, d\xi = (V_{\widehat \phi _2}\widehat f)(\iota , -(j+k)).
$$
Since
\begin{equation*}
(V_{\widehat \phi _2}\widehat f)(\xi ,-x) = e^{i\scal x\xi }V_{\phi _2} f(x,\xi ),
\end{equation*}
we get
$$
(\op (\Phi _{\mabfj , \mabfk})f)(x) = \big (e^{i\scal k\iota} V_{\phi _2} f(j+k,\iota )\big )
\phi _1(x-j)e^{i\scal x{\iota +\kappa }}.
$$

\par

This gives
\begin{multline}\label{OpaComp}
(\op (a)f)(x) = \sum _{\mabfj ,\mabfk \in \Lambda ^2} (V_\Psi a)(j,\iota ,\kappa ,k)
e^{i\scal k\iota} V_{\phi _2} f(j+k,\iota )\phi _1(x-j)e^{i\scal x{\iota +\kappa }}
\\[1ex]
= \sum _{\mabfj ,\mabfk \in \Lambda ^2} (V_\Psi a)(j,\kappa ,\iota -\kappa ,k-j)
e^{i\scal {k-j}\kappa} V_{\phi _2} f(\mabfk )\phi _1(x-j)e^{i\scal x\iota}
\\[1ex]
= \sum _{\mabfj \in \Lambda ^2} h(\mabfj )\phi _1(x-j)e^{i\scal x\iota},
\end{multline}
where
$$
h(\mabfj ) = \sum _{\mabfk \in \Lambda ^2} c_0(\mabfj ,\mabfk )
V_{\phi _2} f(\mabfk ).
$$
The result now follows from the facts that $h=A\cdot (C_{\phi _2}f)$ and that
the right-hand side of \eqref{OpaComp} is equal to $(D_{\phi _1}h)(x)$.
%
%
%
%
%
%
%
%
%
%
%
\end{proof}

\par

\begin{proof}[Proof of Theorem \ref{thmOpCont}]
By Proposition 1.7 in \cite{Toft12} and its proof, it suffices to prove
the result for $t=0$. Let $\omega$, $\omega _0$, $\Lambda$,
$\phi _1$, $\phi _2$ and $A$ be as in Lemma \ref{PsDoDisc}. Then
\begin{equation}\label{SynthOps}
C_{\phi _2}\, :\, M^{\mabfp _1}_{(\omega _1)}(\rr d) \to
\ell ^{\mabfp _1}_{(\omega _1)}(\Lambda ^2)
\quad \text{and}\quad
D_{\phi _1}\, :\, \ell ^{\mabfp _2}_{(\omega _2)}(\Lambda ^2 )
\to M^{\mabfp _2}_{(\omega _2)}(\rr d)
\end{equation}
are continuous.

\par

Furthermore, since
$$
\frac {\omega _2(x,\xi )}{\omega _1(y,\eta )}\lesssim
\omega _0(x,\eta ,\xi -\eta ,y-x),
$$
it follows from \eqref{omega0omegaRel} that
$$
\frac {\omega _2(X)}{\omega _1(Y)}\lesssim
\omega (X ,Y),\quad X=(x,\xi )\in \rr {2d},\ Y=(y,\eta )\in \rr {2d}.
$$
holds. Hence Theorem \ref{matrixcont2} shows that
$$
A \, :\, \ell ^{\mabfp _1}_{(\omega _1)}(\Lambda ^2 ) \to
\ell ^{\mabfp _2}_{(\omega _2)}(\Lambda ^2 )
$$
is continuous. Hence, if $\op (a)$ is defined by
\eqref{OpaFactorization}, it follows that $\op (a)$
from $\maclS _{1/2}(\rr d)$ to $\maclS _{1/2}'(\rr d)$
extends to a continuous map from
$M^{\mabfp _1}_{(\omega _1)}(\rr d)$ to
$M^{\mabfp _2}_{(\omega _2)}(\rr d)$.

\par

It remains to prove the uniqueness of the extension,
If $\max (\mabfp _1 )<\infty$, then the uniqueness follows
from the fact that $\maclS _{1/2}(\rr d)$ is dense in
$M^{\mabfp _1}_{(\omega _1)}(\rr d)$. If instead
$p<\infty$, then $q<\infty$, and $\maclS _{1/2}(\rr {2d})$
is dense in $M^{p,q}_{(\omega _0)}(\rr {2d})$. The
uniqueness now follows in this case from
\eqref{PsDOEst}$'$, and the fact that
$\op (a)$ is uniquely defined as an operator from
$\maclS _{1/2}'(\rr d)$ to $\maclS _{1/2}(\rr d)$, when
$a\in \maclS _{1/2}(\rr {2d})$.

\par

Finally assume that $p=\infty$ and $\max (\mabfp _1 ) =\infty$. Then
\eqref{pqconditions}$'$ shows that $q\le 1$. In particular, if
$f\in M^{\mabfp _1}_{(\omega _1)}(\rr d)$ then $f\in M^\infty _{(\omega _1)}(\rr d)$.
The uniqueness now follows from the fact that $\op (a)f$ is uniquely
defined as an element in $M^\infty _{(\omega _2)}(\rr d)$, in view of
Theorem A.2 in \cite{Toft11}.
\end{proof}

\par

We have also the following result on Schatten-von Neumann properties
for pseudo-differential operators.

\par

\begin{thm}\label{thmOpSchatten}
Let $t\in \mathbf R$, $\omega _1,\omega _2\in \mathscr P_{E}(\rr {2d})$
and $\omega _0\in \mathscr P_{E}(\rr {2d}\oplus \rr {2d})$
be such that
\begin{equation}\label{Eq:omegajCond}
\frac {\omega _2(x,\xi  )}{\omega _1
(y,\eta )} \asymp \omega _0( (1-t)x+ty,t\xi +(1-t)\eta ,\xi -\eta ,y-x )
\end{equation}
Also let 
$p,p_j,q,q_j\in (0,\infty ]$ be such that
$$
p_1\le p,\quad q_1\le \min (p,p'),\quad p_2\ge \max (p,1),
\quad q_2\ge \max (p,p').
$$
Then
\begin{equation}\label{ModSchattenEmb}
M^{p_1,q_1}_{(\omega _0)}(\rr {2d})
\subseteq
s_{t,p}(\omega _1,\omega _2)
\subseteq
M^{p_2,q_2}_{(\omega _0)}(\rr {2d})
\end{equation}
and
\begin{equation}\label{PsDOEstSchatt}
\nm a{M^{p_2,q_2}_{(\omega _0)}}\lesssim
\nm a{s_{t,p}(\omega _1,\omega _2)}\lesssim
\nm a{M^{p_1,q_1}_{(\omega _0)}}.
\end{equation}
\end{thm}

\par

\begin{proof}
We use the same notations as in the proof of Theorem \ref{thmOpCont}.
The result is true for $p\in [1,\infty ]$ in view of Theorem A.3 in
\cite{Toft11} and Proposition \ref{p1.4A}. Hence it suffices to prove the
assertion for $p\in (0,1)$.

By Proposition 1.7 in \cite{Toft12} and its proof, it suffices
to prove the result for $t=0$. 

\par

It follows from \eqref{YoungSchatten} and
\eqref{OpaFactorization} that
\begin{multline*}
\nm {\op (a)}{\mathscr I_p(\omega _1, \omega _2)}
=
\nm {D_{\phi _1} \circ A \circ C_{\phi _1}}{\mascI _p(\omega _1, \omega _2)}
\\[1ex]
\lesssim
\nm {D_{\phi _1}}
{\mascI _\infty {(\ell ^2_{(\omega _2)}(\Lambda ^2 ),M^2_{(\omega _2)})}}
\nm A{\mascI _p {(\ell ^2_{(\omega _1)}(\Lambda ^2 ),
\ell ^2_{(\omega _2)}(\Lambda ^2 ) ) }}
\nm {C_{\phi _2}}
{\mascI _\infty {(M^2_{(\omega _1)},\ell ^2_{(\omega _1)} (\Lambda ^2 )  )}}
\\[1ex]
\asymp 
\nm A{\mascI _p {(\ell ^2_{(\omega _1)}(\Lambda ^2 ),
\ell ^2_{(\omega _2)}(\Lambda ^2 ) ) }}
\lesssim \nm A{\mathbb U^{p,p}(\omega ,\Lambda ^2  )}
\asymp \nm a{M^{p,p}_{(\omega _0)}},
\end{multline*}
and the result follows.
%
%
\end{proof}

\par

\begin{rem}
Theorems \ref{thmOpCont} and \ref{thmOpSchatten} are related
to certain results \cite{MoPf,Pf} when the involved weights are trivial,
and the involved Lebesgue exponents belong the subset $[1,\infty ]$ of
$(0,\infty ]$. More precisely, let $\mabfp \in [1,\infty ]^{4d}$
be given by
$$
\mabfp =(p_1,\dots ,p_1,p_2,\dots ,p_2,q_1,\dots ,q_1,q_2,\dots ,q_2),
$$
and each $p_j$ and $q_j$ occur $d$ times. Then
S. Molahajloo and G. E. Pfander
investigate in \cite{MoPf,Pf}, continuity of pseudo-differential operators
with symbols in $M^{\mabfp}(\rr {2d})$, when acting between
$M^{r_1,s_1}(\rr d)$ and $M^{r_2,s_2}(\rr d)$, for some $r_j,s_j\in [1,\infty ]$
(cf. Theorem 1.3 in \cite{MoPf}).

\par

We note that there are some overlaps between Theorems \ref{thmOpCont}
and \ref{thmOpSchatten} and the results in \cite{MoPf,Pf}. On the other hand,
the results in \cite{MoPf,Pf}, and Theorems \ref{thmOpCont}
and \ref{thmOpSchatten} do not contain each others, since the assumptions
on the symbols are more restrictive in Theorems \ref{thmOpCont}
and \ref{thmOpSchatten}, while the assumptions on domains and image
spaces are more restrictive in \cite{MoPf,Pf}.
\end{rem}

\par

Next we show that Theorem \ref{thmOpSchatten} is optimal
with respect to $p$. More precisely, we have
the following result

\par

\begin{thm}\label{SchattenConverse}
Let $t\in \mathbf R$, $\omega _k\in \mathscr P_{E}(\rr {2d})$,
$k=1,2$, and $\omega _0\in \mathscr P_{E}(\rr {2d}\oplus \rr {2d})$
be such that \eqref{Eq:omegajCond} holds.
Also let $p,q,r\in (0,\infty ]$, and suppose
\begin{equation}\label{SchattenModIncl}
M^{p,q}_{(\omega _0)}(\rr {2d})\subseteq s_{t,r}
(\omega _1, \omega _2).
\end{equation}
Then the following is true:
\begin{enumerate}
\item $p\le r$ and $q\le \min (2,r)$;

\vrum

\item if in addition $\omega _1,\omega _2 \in \mascP (\rr {2d})$
and $r\ge 2$, then $q\le p'$.
\end{enumerate}
\end{thm}

\par

We need some preparations for the proof.
The following result concerning Wigner distributions
extends \cite[Proposition A.4]{Toft11} to involve Lebesgue
exponents smaller than one (cf. \eqref{wignertdef}).
We omit the proof
since the arguments are the same as in the proof of
\cite[Proposition A.4]{Toft11}. (See also \cite{Gc2,Toft5} and the references
therein for related results.)

\par

\begin{prop}\label{t-WignerMod}
Let $t\in \mathbf R$, and let $p_j,q_j,p,q\in (0,\infty ]$ be such that
$p\le p_j,q_j\le q$, for $j=1,2$, and 
\begin{equation}\label{(A.8)}
1/p_1+1/p_2=1/q_1+1/q_2=1/p+1/q.
\end{equation}
Also let $\omega _1,\omega _2\in \mascP _E(\rr {2d})$ and
$\omega \in \mascP _E(\rr {2d}\oplus \rr {2d})$ be such that
\begin{equation}\label{(A.9)}
\omega _0( (1-t)x+ty,t\xi +(1-t)\eta ,\xi -\eta ,y-x )
\lesssim
\omega _1(x,\xi )\omega _2(y,\eta ).
\end{equation}
Then the map $(f_1,f_2)\mapsto W_{f_1,f_2}^t$ from $\maclS _{1/2}'(\rr
d)\times \maclS _{1/2}'(\rr d)$ to $\maclS _{1/2}'(\rr {2d})$ restricts to a
continuous mapping from $M^{p_1,q_1}_{(\omega _1)}(\rr d)\times
M^{p_2,q_2}_{(\omega _2)}(\rr d)$ to $M^{p,q}_{(\omega _0
)}(\rr {2d})$, and
\begin{equation}\label{(A.10)}
\nm {W_{f_1,f_2}^t}{M^{p,q}_{(\omega _0)}}
\lesssim
\nm {f_1}{M^{p_1,q_1}_{(\omega _1)}} \nm
{f_2}{M^{p_2,q_2}_{(\omega _2)}}
\end{equation}
when $f_1,f_2\in \maclS _{1/2}'(\rr d)$.
\end{prop}

\par

We have now the following extension of Corollary 4.2 (1) in \cite{Toft2}.

\par

\begin{cor}\label{cor4.2}
Let $p\in (0,\infty ]$, $q\in (2,\infty ]$, $t\in \mathbf R$, and let
$\omega _2\in \mascP _E(\rr {2d})$ and
$\omega _0\in \mascP _E(\rr {4d})$ be such that
$$
\omega _0((1-t)x,t\xi ,\xi ,-x)\lesssim \omega _2(x,\xi ).
$$
Then there is an element $a$ in $M^{p,q}_{(\omega _0)}(\rr {2d})$
such that $\op _t(a)$ is not continuous from $\maclS _{1/2}(\rr d)$
to $M^{2,2}_{(\omega _2)}(\rr d)$.
\end{cor}

\par

\begin{proof}
Let $a=W ^t_{f_2,f_1}$,
where $f_1\in \Sigma _1(\rr d)\setminus 0$ and $f_2\in
M^{q,p}_{(\omega _2)}(\rr d)\setminus M^{2,2}_{(\omega _2)}(\rr {2d})$.
Such choices of $f_2$ are possible in view of Proposition
\ref{ConseqThmS}.

\par

By using the fact that $\omega _0$ and $\omega _2$ are moderate
weights, it follows that \eqref{(A.9)} holds when $\omega _1(x,\xi )
=e^{c(|x|+|\xi |)}$, and the constant $c>0$ is chosen large enough.
By Proposition \ref{p1.4A}, it follows that $f_1\in
M^{p,q}_{(\omega _1)}(\rr d)$. Hence
$a\in M^{p,q}_{(\omega _0)}(\rr {2d})$ in view of Proposition
\ref{t-WignerMod}.

\par

On the other hand, if $f\in \maclS _{1/2}(\rr d)\setminus 0$ is chosen
such that $f$ and $f_1$ are not orthogonal, then
$$
\op _t(a)f = (f,f_1)\cdot f_2\in
M^{q,p}_{(\omega _2)}(\rr d)\setminus M^{2,2}_{(\omega _2)}(\rr {2d}),
$$
and the result follows.
\end{proof}

\par

We also need the following lemma. We omit the proof since the
result is a special case of Proposition 4.3 in \cite{Toft12}. Here
$\check f$ is defined as $\check f(x) = f(-x)$ when $f\in \maclS _{1/2}'(\rr d)$,
and recall from Subsection \ref{subsec1.8} that
$(\mascF _\sigma a)(X) = 2^d\widehat a(-2\xi ,2x)$ when $X=(x,\xi )
\in \rr {2d}$.

\par

\begin{lemma}\label{identification1}
Let $\omega _1,\omega _2\in \mascP _E(\rr {2d})$, $a\in
\mathscr S'(\rr{2d})$, and that $p\in (0,\infty ]$. Then
$$
\mascF _\sigma (s_p^w(\omega _1,\omega _2))
=
s_p^w(\omega _1,\check \omega _2).
$$
\end{lemma}

\par

\begin{proof}[Proof of Theorem \ref{SchattenConverse}]
We may assume that $t=1/2$, and consider first the case when $1\le r$.
Let $\splM ^{p,q}_{(\omega )}(\rr {2d})$ and $\splW ^{p,q}_{(\omega )}(\rr {2d})$
be the modulation spaces when the symplectic Fourier transform
is used instead of the ordinary Fourier transform in the definition of
modulation space norms of $M^{p,q}_{(\omega )}(\rr {2d})$ and
$W^{p,q}_{(\omega)}(\rr {2d})$, respectively.
Then \eqref{SchattenModIncl} is equivalent to
$$
\op ^w(\splM ^{p,q}_{(\omega )}) \subseteq \mascI _r
(\omega _1,\omega _2),
\quad \text{when}\quad 
\frac {\omega _2(X-Y)}{\omega _1(X+Y)} \asymp \omega (X,Y)
$$
(see e.{\,}g. \cite{Toft11}).
Let $\phi \in \Sigma (\rr {2d})$ and $\Lambda$ in \eqref{LambdaDef}
be chosen such that $\{ \phi (  \cdo -\mabfj )e^{-2i\sigma (\cdo ,\mabfk )}\}
_{\mabfj ,\mabfk \in \Lambda ^2}$ is a Gabor frame. Also let
$\vartheta (\mabfk )=\omega (0,\mabfk )$,
$c\in \ell ^\infty _{(\vartheta )}(\Lambda ^2)$, $c_0(0,\mabfk )
=c(\mabfk )$ and $c_0(\mabfj ,\mabfk ) = 0$ when $\mabfj \neq 0$,
and let
$$
a(X) \equiv \sum _{\mabfk \in \Lambda ^2}c(\mabfk )\phi (X)
e^{-2i\sigma (X,\mabfk )}
=
\sum _{\mabfj ,\mabfk \in \Lambda ^2}c_0(\mabfj ,\mabfk )
\phi (X-\mabfj )e^{-2i\sigma (X,\mabfk )}.
$$
Then $a\in \splM ^{p,\infty}_{(\omega )}(\rr {2d})$
for every $p\in (0,\infty ]$. Furthermore,
\begin{equation}\label{aDiffModEq}
a\in \splM ^{p,q}_{(\omega )}
\quad \Longleftrightarrow \quad
a\in \splW ^{p,q}_{(\omega )}
\quad \Longleftrightarrow \quad
c\in \ell ^q_{(\vartheta )}
\end{equation}
holds for every $p,q\in (0,\infty ]$.

\par

Now if $q>r$, then choose $c\in \ell ^q_{(\vartheta )}\setminus \ell
^r_{(\vartheta )}$, and it follows from \eqref{ModSchattenEmb}
and \eqref{aDiffModEq} that $a\in \splM ^{p,q}_{(\omega )}
\setminus s^w_r(\omega _1,\omega _2)$. This shows
that $q\le r$ when \eqref{SchattenModIncl} holds.

\par

Assume instead that $p>r$, let $q\in (0,\infty ]$ be arbitrary, choose
$c\in \ell ^p_{(\vartheta )}\setminus \ell ^r_{(\vartheta )}$, and consider
$$
b=\mathscr F_\sigma a \in \mathscr F_\sigma \splW ^{q,p}_{(\omega )}
= \splM^{p,q}_{(\omega _{T})},\qquad \omega _{T}(X,Y)= \omega (Y,X).
$$
By
Lemma \ref{identification1}, \eqref{ModSchattenEmb} and \eqref{aDiffModEq}
it follows that $b\notin s^w_r(\omega _1,\check \omega _2)$.
This shows that $p\le r$ when \eqref{SchattenModIncl} holds,
and the result follows in the
case $r\ge 1$.

\par

Next assume that $r<1$. If \eqref{SchattenModIncl} holds for some
$q>r$, then it follows by (real) interpolation between the cases
\eqref{SchattenModIncl} and
$$
\op ^2(\splM ^{2,2}_{(\omega )}(\rr {2d}))= \mascI _2(M^2_{(\omega _1)}(\rr d),
M^2_{(\omega _2)}(\rr d))
$$
that \eqref{SchattenModIncl} holds for $r=1$ and some $q>1$. This contradicts
the first part of the proof. If instead \eqref{SchattenModIncl} holds for some
$p>r$, then it again follows by interpolation that \eqref{SchattenModIncl}
holds for $r=1$ and some $p>1$, which again contradicts
the first part of the proof. This shows that $p,q\le r$ if \eqref{SchattenModIncl}
should hold. Furthermore, by Corollary \eqref{cor4.2} it follows that
$q\le 2$ when \eqref{SchattenModIncl} holds, and (1) follows.

\par

It remains to prove (2). By \cite[Corollary 3.5]{GH1} it follows that the result
is true for trivial weights in the modulation space norms, and the result
is carried over to the case with non-trivial weights by using lifting properties,
established in \cite{GrochToft1}. The proof is complete.
\end{proof}

\par

\section{Applications to the H{\"o}rmander-Weyl calculus}
\label{sec4}

\par

In this section we apply the results in the previous section to deduce
Schatten-von Neumann properties in the H{\"o}rmander-Weyl calculus. (See
e.{\,}g. \cite{Le,Toft4}, Sections 18.4--18.6 in \cite{Ho1} and Subsection
\ref{subsec1.8} for approaches or notations.)

\par

The following result is a consequence of
Theorem 4.4 (1) in \cite{Toft4} in the case $p\ge 1$, while
the latter result do not touch the case when $p<1$.

\par

\begin{thm}\label{thm:WeylHorm1}
Let $p\in (0,2]$, $g$ be feasible on $W$, and let $m$ be a
$g$-continuous weight on $W$
such that $m\in L^p(W)$. Then $S(m,g)\subseteq s_p^w(W)$.
\end{thm}

\par

%

We need some preparations for the proofs. First we recall that
for any feasible metric $g$ and any $X\in W$, there are symplectic
coordinates, and numbers
$$
0<\lambda _d(X)\le \dots \le \lambda _1(X)\le 1
$$
such that
$$
g_X(Y) = \sum _{k=1}^d \lambda _k(X)(y_k^2+\eta _k^2),
\quad
g_X^\sigma (Y) = \sum _{k=1}^d \lambda _k(X)^{-1}(y_k^2+\eta _k^2),
$$
where $Y=(y,\eta )\in W$ in these coordinates. The intermediate metric
$$
g^{{}_0}_X(Y) = \sum _{k=1}^d (y_k^2+\eta _k^2)
$$
is symplectically invariant defined and is \emph{symplectic}, i.{\,}e.
$(g^{{}_0})^\sigma =g^{{}_0}$.

\par

We have the following lemma.

\par

\begin{lemma}\label{prop:sumOpsSchatten}
Let $p\in (0,1]$, $\mascH _1$ and $\mascH _2$ be Hilbert spaces,
and let $T_j\in \maclB (\mascH _1,\mascH _2)$,
$j\ge 1$. Then
\begin{equation*}
\nm T{\mascI _p(\mascH _1,\mascH _2)}
\le
\left (\sum _{k=1}^\infty \nm {T_k}{\mascI _p(\mascH _1,\mascH _2)}^p \right )^{1/p},
\qquad
T=\sum _{k=1}^\infty T_k ,
\end{equation*}
provided the right-hand side makes sense as an element in
$\maclB (\mascH _1,\mascH _2)$.
\end{lemma}

\par

We refer to \cite[Appendix 1.1]{Pe} for the proof of Lemma
\ref{prop:sumOpsSchatten}.

\par

\begin{cor}\label{prop:sumOpsSchatten2}
Let $p\in (0,1]$, $t\in \mathbf R$ and $a\in \mascS '(\rr {2d})$.
Also let $\{ÊÊ\fy _j Ê\} _{j\in I}$ be a sequence in $\mascS  (\rr {2d})$
such that $0\le \fy _j\le 1$ for every $j$ and
$$
\sum _{j\in I}\fy _j =1.
$$
Then
\begin{equation}\label{stpSumEst}
\nm a{s_{t,p}(\rr {2d})}
\le
\left (  \sum _{j\in I}  \nm {\fy _ja}{s_{t,p}(\rr {2d})}^p \right )^{1/p}.
\end{equation}
\end{cor}

\par

\begin{proof}
Let $a_j=\fy _ja$, $I_k\subseteq I$, $k\ge 1$ be a sequence of increasing
and finite sets such that $\bigcup _{k=1}^\infty I_k=I$, and set
$b_k=\sum _{j\in I_k}a_j$.
We may assume that
\begin{equation}\label{Eq:SumajEst}
\sum _{j\in I}  \nm {a_j}{s_{t,p}}^p<\infty ,
\end{equation}
since otherwise there is nothing to prove.

\par

If $k_1\le k_2$, then Lemma \ref{prop:sumOpsSchatten} gives
\begin{equation}\label{Eq:bkEsts}
\nm {b_{k_2}-b_{k_1}}{s_{t,p}}^p \le \sum _{j\in I_{k_2}\setminus I_{k_1}}
\nm {a_j}{s_{t,p}}^p
\quad\text{and}\quad
\nm {b_k}{s_{t,p}}^p \le \sum _{j\in I_k} \nm {a_j}{s_{t,p}}^p.
\end{equation}
By \eqref{Eq:SumajEst} we get $\nm {b_{k_2}-b_{k_1}}{s_{t,p}}\to 0$
as $k_1,k_2\to \infty$, and by completeness, there is a unique element
$b\in {s_{t,p}}$ such that $\nm {b_k-b}{s_{t,p}}\to 0$ as $k\to \infty$.

\par

Since $s_{t,p}\subseteq s_{t,2}=L^2$, we get $b\in L^2$ and $b_k\to b$
in $L^2$ as well when $k\to \infty$. Furthermore, since $|b_k|\le |a|$
and $b_k\to a$ pointwise, Beppo Levi's theorem gives that
$a\in L^2$ and $a=b$. The result now follows by letting $k$ tends to
$\infty$ in \eqref{Eq:bkEsts}.
\end{proof}

\par

\begin{proof}[Proof of Theorem \ref{thm:WeylHorm1}]
Since $S(m,g)\subseteq
S(m,g^{{}_0})$ when $g^{{}_0}$ is the symplectic metric of
$g$, and that $m$ is $g^{{}_0}$-continuous when $m$
is $g$-continuous, we may assume that $g$ is symplectic.
(Cf. \cite{Toft4}.)

\par

Let $c>0$, $C>0$, $I$, $X_j$, $U_j$ and $\fy _j$ be chosen as in
Remarks 2.3 and 2.4 in \cite{Toft4}, except that we may assume
that $c>0$ and $C>0$ was chosen such that
\begin{gather*}
C^{-1}m(X)\le m(Y)\le Cm(X)
\quad \text{and}\quad
C^{-1}g_X\le g_Y\le Cg_X
\\[1ex]
\text{when}\quad g_Y(X-Y)<2c
\end{gather*}
Also set $a_j=\fy _ja$. Then $\supp a_j\subseteq U_j$,
and the set of $a_j$ is bounded in $S(m,g)$.
We first estimate $\nm {a_j}{s_p^w}$ for a fixed $j\in I$.

\par

Let $B_r(X_0)$ denotes the open ball with center at $X_0$ and radius $r$,
$\Phi \in C_0^\infty (B_1(0))\setminus 0$ be such that $0\le \Phi \le 1$,
and choose symplectic coordinates $X=(x,\xi )\in \rr d\times \rr d\simeq
\rr {2d}$ such that $g_{X_j}$ takes the form
$$
g_{X_j}(X) = g_{X_j}(x,\xi ) = |X|^2=|x|^2 + |\xi |^2,
$$
in these coordinates. Then for every multi-index $\alpha$, we have
$\partial _g^\alpha = \partial _x^{\alpha _1}\partial _\xi ^{\alpha _2}$ for some
$\alpha _1$ and $\alpha _2$ such that $\alpha _1+\alpha _2=\alpha$, and the
support of $a_j$ is contained in $B_c(X_j)$. Let $N\ge 0$ be an integer such that
$Np>d$. By Theorem 3.1 in \cite{GaSa} and
Theorem \ref{thmOpSchatten} we get
\begin{multline}\label{ajSchattenEst1}
\nm {a_j}{s_p^w}^p \le C_1\iiiint | V_\Phi a_j(x,\xi ,\eta ,y)|^p\, dxd\xi d\eta dy
\\[1ex]
\le C_2\sum _{|\alpha +\beta|=2N} \iiiint | V_{\partial ^\beta \Phi} (\partial
^\alpha a_j)(x,\xi ,\eta ,y)\eabs {(y,\eta )}^{-2N}|^p\, dxd\xi d\eta dy
\\[1ex]
\le C_3\sum _{|\alpha \le 2N} \iiiint | V_\Phi (\partial
^\alpha a_j)(x,\xi ,\eta ,y)|^p\eabs {(y,\eta )}^{-2Np}\, dxd\xi d\eta dy,
\end{multline}
for some constants $C_1,\dots ,C_3$ which are independent of $j\in I$.

\par

We shall estimate the last integrand in \eqref{ajSchattenEst1}. Let $\chi _{0,j}$
be the characteristic function of $U_j=B_c(X_j)$, and let $\chi _j$
be the characteristic function of $U_j=B_{1+c}(X_j)$. Then
$$
|(\partial ^\alpha a_j)(X)|\le C_{1,\alpha}m(X)\chi _{0,j}(X)
\le C_{2,\alpha}m(X_j)\chi _{0,j}(X),
$$
for some constants $C_{1,\alpha}$ and $C_{2,\alpha}$, which only depend
on $\alpha$. Here the last step follows from the fact that $m$ is
$g$-continuous. This gives
\begin{multline*}
| V_\Phi (\partial ^\alpha a_j)(x,\xi ,\eta ,y)|
\le
(2\pi )^{-d/2}\int _{\rr {2d}} |(\partial ^\alpha a_j)(Z-X)|\Phi (Z)\, dZ
\\[1ex]
\le
C_{3,\alpha}m(X_j)\int _{\rr {2d}} |\chi _{0,j}(Z-X)|\Phi (Z)\, dZ
\le
C_{4,\alpha}m(X_j)\chi _j(X)
\\[1ex]
=
C_{5,\alpha}\left (m(X_j)^p\int \fy _j(Z)\, dZ \right )^{1/p} \chi _j(X)
\\[1ex]
\le
C_{6,\alpha}\left ( \int m(Z)^p\fy _j(Z)\, dZ \right )^{1/p} \chi _j(X),
\end{multline*}
for some constants $C_{3,\alpha},\dots ,C_{6,\alpha}$, which only depend
on $\alpha$.

\par

By combining the last estimate with \eqref{ajSchattenEst1}, we get
\begin{multline*}
\nm {a_j}{s_p^w}^p\le C_1 \left (  \iint \chi _j(X)^p\eabs Y^{-2Np}\, dXdY  \right )
\cdot  \left ( \int m(Z)^p\fy _j(Z)\, dZ  \right )
\\[1ex]
\le
C_2\int m(Z)^p\fy _j(Z)\, dZ,
\end{multline*}
for some constants $C_1$ and $C_2$, which are independent of $j$.

\par

A combination of the last estimate and Corollary
\ref{prop:sumOpsSchatten2} now gives
$$
\nm a{s_p^w}^p \le \sum _j \nm {a_j}{s_p^w}^p
\lesssim \sum _j \int m(Z)^p\fy _j(Z)\, dZ = \nm m{L^p}^p,
$$
and the result follows.
\end{proof}

\par

By similar arguments as in the proof of \cite[Theorem 2.11]{Toft4},
Theorem \ref{thm:WeylHorm1} gives the following results involving
suitable Sobolev type spaces, introduced in by Bony and Chemin
in \cite{BoCh} (see also \cite{Le}). The details are left for the reader. Here recall that
the Riemannian metric $g$ on $W$ is called \emph{split} if there are symplectic
coordinates such that
$$
g_X(y,\eta ) = g_X(y,-\eta ),\qquad X\in W,\ Y=(y,\eta )\in W,
$$
in these coordinates.

\par

\begin{thm}\label{thm:WeylHorm3}
Let $t=1/2$, $a\in \mascS '(W)$, $p\in (0,\infty ]$, $g$ be strongly
feasible on $W$, and let $m,m_1,m_2$ be $g$-continuous and
$(\sigma ,g)$-temperate weights on $W$ such that $m_2m/m_1\in L^p(W)$.
Then $S(m,g)\subseteq s_{t,p}(m_1,m_2,g)$.

\par

Furthermore, if in addition $g$ is split, then $S(m,g)\subseteq s_{t,p}(m_1,m_2,g)$
holds for general $t\in \mathbf R$.
\end{thm}

\par

\begin{rem}
We note that \cite[Theorem 2.11]{Toft4} covers Theorem \ref{thm:WeylHorm3}
when $p\ge 1$. We also note that the assumption $a\in S(m,g)$ is missing, and
that $s_{t,p}^w$ and $s_{t,\sharp}^w$ should be $s_{t,p}$ and $s_{t,\sharp}$,
respectively, in (3) and (4) in \cite[Theorem 2.11]{Toft4}.
\end{rem}

\par

\section{Applications to compactly supported Schatten-von Neumann
symbols}\label{sec5}

\par

In this section we introduce a subset $s_{t,p}^q(\rr {2d})$ of $s_{t,p}(\rr {2d})$
when $p,q\in (0,\infty ]$. We show that
the set of compactly supported
elements in $s_{t,p}^q(\rr {2d})$ with $q=\min (p,1)$ and $q=1$
is a subspace and superspace, respectively, of compactly
supported elements in $\mathscr F L^p$. This
result goes back to \cite{Toft0,Toft1} in the case $p\ge 1$ and $t=1/2$.
The proof is based on Theorem \ref{thmOpSchatten} and certain
characterizations given here, which might be of independent interests.

\par

First we make the following definition. Here $\ON _d$ is the set of
all orthonormal sequences in $L^2(\rr d)$.

\par

\begin{defn}
Let $t\in \mathbf R$ and $p,q\in (0,\infty ]$. Then $s_{t,p}^q(\rr {2d})$
consists of all $a\in \mascS '(\rr {2d})$ of the form
$$
a=\sum _{j=0}^\infty \lambda _j W^t_{f_j,g_j},
$$
for some non-negative and non-increasing sequence
$\{\lambda _j \} _{j=0}^\infty \in \ell ^p(\mathbf N)$, and
some $\{ f_j \} _{j=0}^\infty \in \ON _d$ and $\{ g_j \}
_{j=0}^\infty \in \ON _d$ which at the same time are bounded
in $M^{2q}(\rr d)$. For any $a\in s_{t,p}^q(\rr {2d})$ we set
$$
\nm a{s_{t,p}^q} \equiv \nm {\{ \lambda _j \nm {f_j}{M^{2q_0}}
\nm {g_j}{M^{2q_0}} \} _{j=0}^\infty }{\ell ^p},\qquad q_0=\min (1,q).
$$
\end{defn}

\par

In the following lemma we present some basic facts for $s_{t,p}^q(\rr {2d})$.

\par

\begin{lemma}\label{basicLemmaspp}
Let $p,q,p_j,q_j\in (0,\infty ]$, $j=0,1,2$ be such that $p_1\le p_2$,
$q_1\le q_2$ and $q_0\ge 1$. Then the following is true:
\begin{enumerate}
\item $s_{t,p}^{q_0}(\rr {2d})=s_{t,p}(\rr {2d})$ and $\nm a{s_{t,p}^{q_0}}
=\nm a{s_{t,p}}$ when $a \in s_{t,p}(\rr {2d})$;

\vrum

\item $\mascS (\rr {2d})\subseteq s_{t,p}^q(\rr {2d})$;

\vrum

\item $s_{t,p_1}^{q_1}(\rr {2d})\subseteq s_{t,p_2}^{q_2}(\rr {2d})$;

\vrum

\item if in addition $t=1/2$, then $a\in s_{t,p}^{q}(\rr {2d})$, if and only if
$\mascF _\sigma a\in s_{t,p}^{q}(\rr {2d})$, and
$$
\nm {\mascF _\sigma a}{s_{t,p}^q} = \nm a{s_{t,p}^q} .
$$
\end{enumerate}
\end{lemma}

\par

\begin{proof}
The assertion (1) follows from the fact that $L^2(\rr d)$ is continuously
embedded in $M^{2q_0}(\rr d)$, since $2q_0\ge 2$, (2) follows from
Proposition \ref{SchattenExp} below, and (3) is a straight-forward
consequence of the definitions.

\par

Finally, (4) follows from the facts that $\mascF _\sigma W_{f,g}=W_{\check f,g}$
and that $\nm {\check f}{M^{2q}}\asymp \nm f{M^{2q}}$. This gives the result.
\end{proof}

\par

The following result characterizes the set of compactly supported
elements in $s_{t,p}(\rr {2d})$.

\par

\begin{thm}\label{CompSuppSchatten}
Let $t\in \mathbf R$ and $p,q\in (0,\infty ]$ be such that
$r\le q$. Then
\begin{multline}\label{CompSuppSchattenEq}
s^q_{t,q}(\rr {2d})\bigcap \mathscr E'(\rr {2d})
\subseteq \mathscr FL^q(\rr {2d}) \bigcap \mathscr E'(\rr {2d})
\\[1ex]
= M^{p,q}(\rr {2d})\bigcap \mathscr E'(\rr {2d})
\subseteq s_{t,q}(\rr {2d})\bigcap \mathscr E'(\rr {2d}).
\end{multline}
\end{thm}

\par

We need some preparations for the proof. First we recall the following facts for the
harmonic oscillator $H=H_d=|x|^2-\Delta$ on $\mascS '(\rr d)$. We omit the proof
since the result is a special case of \cite[Theorem 3.5]{BoTo}.

\par

\begin{lemma}\label{LemmaHarmOscBij}
Let $p,q\in [1,\infty ]$ and let $\omega ,\vartheta _N\in \mascP (\rr {2d})$ be such
that $\vartheta _N(x,\xi )= \eabs {(x,\xi )}^N$, where $N$ is an integer. Then
$H_d^N$ is a homeomorphism on $\mascS (\rr d)$, on $\mascS '(\rr d)$,
and from $M^{p,q}_{(\vartheta _{2N} \omega )}(\rr d)$ to $M^{p,q}_{(\omega )}(\rr d)$.
\end{lemma}

\par

The next result concerns Schwartz kernels of linear operators.

\par

\begin{prop}\label{SchwartzKernels}
Let $T$ be a linear and continuous operator from
$L^2(\rr {d_1})$ to $L^2(\rr {d_2})$ and such that
the kernels of $T^*\circ T$ and $T\circ T^*$
belong to $\mascS (\rr {2d_1})$ and
$\mascS (\rr {2d_2})$, respectively. Then the kernel
of $T$ belongs to $\mascS (\rr {d_2+d_1})$.
\end{prop}

\par

The result should be available in the literature. In order to be
self-contained we here present a
proof, obtained in collaboration with A. Holst
at Lund University, Sweden.

\par

\begin{proof}
Let $N\in \mathbf N$, $\omega _{j,r}(x,\xi )= \eabs {(x,\xi )}^{r}$ when $x,\xi \in
\rr {d_j}$ and $r\in \mathbf R$, and let $f_j\in \mascS
(\rr {d_j})$, $j=1,2$.
Then $H^N\circ S\circ H^N$ is an operator with kernel in $\mascS$, if and only if $S$
is an operator with kernel in $\mascS$.

\par

By the assumptions and Lemma \ref{LemmaHarmOscBij} we obtain
\begin{align*}
\nm {Tf_1}{L^2(\rr {d_2})}^2 &= ((T^*\circ T)f_1,f_1)_{L^2(\rr {d_1})}
\lesssim \nm {f_1}{M^2_{(\omega _{1,-2N})}}^2
\intertext{and}
\nm {T^*f_2}{L^2(\rr {d_1})}^2 &= ((T\circ T^*)f_2,f_2)_{L^2(\rr {d_2})}
\lesssim \nm {f_2}{M^2_{(\omega _{2,-2N})}}^2.
\end{align*}
Hence
$$
T\in \maclB (M^2_{(\omega _{1,-2N})}(\rr {d_1}),
L^2(\rr {d_2}))
\quad \text{and} \quad T^* \in \maclB (M^2_{(\omega _{2,-2N})}
(\rr {d_2}),L^2(\rr {d_1})).
$$

\par

By duality it also follows that $T\in \maclB (L^2(\rr {d_1}),
M^2_{(\omega _{2,2r})}(\rr {d_2}))$, since the dual of
$M^2_{(\omega _{j,-N})}(\rr {d_j})$ equals $M^2_{(\omega _{j,N})}(\rr {d_j})$
when the $L^2$ form is used (cf. e.{\,}g. \cite[Theorem 11.3.6]{Gc2}).
%

\par

By interpolation between these results we get
$$
T\in \maclB (M^2_{(\omega _{1,-N})}(\rr {d_1}),
M^2_{(\omega _{2,N})}(\rr {d_2})).
$$
and since
$$
\bigcap _{r\in \mathbf R} M^2_{(\omega _{j,N})}(\rr {d_j})
=\mascS (\rr {d_j})
\quad \text{and}\quad
\bigcup _{r\in \mathbf R} M^2_{(\omega _{j,r})}(\rr {d_j})
=\mascS '(\rr {d_j})
$$
when $j=1,2$ also in topological sense in view of e.{\,}g.
\cite {Toft3}, we obtain $T\in \maclB (\mascS '(\rr {d_1}) , \mascS (\rr {d_2}))$.
This implies that the kernel of $T$ belongs to $\mascS (\rr {d_2+d_1})$.
\end{proof}

\par

The next result extends in several ways Lemma 4.1.2 in \cite{Toft0} and
concerns suitable Wigner distribution expansions (see \eqref{wignertdef}).

\par

\begin{prop}\label{SchattenExp}
Let $a\in \mascS (\rr {2d})$, $t\in \mathbf R$,
$p,q\in (0,\infty ]$ and let $H=|x|^2-\Delta$ be the harmonic oscillator on
$\rr d$. Then
\begin{equation}\label{SymbolSpectralExpansion}
a=\sum _{j=0}^\infty \lambda _j W^t_{f_j,g_j},
\end{equation}
for some non-negative and non-increasing
$\{\lambda _j\} _{j=0}^\infty \subseteq
\mathbf R$, $\{ f_j \} _{j=0}^\infty \in \ON (L^2(\rr d))$
and $\{ g_j \} _{j=0}^\infty \in \ON (L^2(\rr d))$ such that
\begin{gather}
\lambda _j \ge 0,\quad f_j,g_j\in \mascS (\rr d),\quad
j\ge 0,\notag
\intertext{and}
\sum _{j=0}^\infty \lambda _j^p \nm {H^Nf_j}{M^q}
\nm {H^Ng_j}{M^q}<\infty ,
\quad \text{when}\quad
N\ge 0.\label{SymbSpecExpEst}
\end{gather}
\end{prop}

\par

\begin{proof}
By the spectral theorem of compact operators,
\eqref{SymbolSpectralExpansion} holds true for
some non-negative and non-increasing sequence
$\{\lambda _j\} _{j=0}^\infty$, and some
$\{ f_j \} _{j=0}^\infty$ and $\{ g_j \} _{j=0}^\infty$ in
$\ON _d$. Since
$$
e^{i(t-1/2)\scal {D_\xi}{D_x}}W^t_{f,g}=W_{f,g}
$$
and $e^{i(t-1/2)\scal {D_\xi}{D_x}}$ is continuous on $\mascS (\rr{2d})$,
it suffices to consider the Weyl case, $t=1/2$.

\par

First we assume that $T\equiv \op ^w(a)\ge 0$, giving that
\eqref{SymbolSpectralExpansion} holds with $g_j=f_j$.
The result is true for $p=1$ and $q=2$ in view of \cite[Lemma 4.1.2]{Toft0}.
Since the kernel of $T$ belongs to $\mascS$, Proposition
\ref{SchwartzKernels} shows that the kernel of $T_N = T^{1/2N}$ belongs
to $\mascS$ for every $N\ge 1$. Furthermore, if $a_N$ is the Weyl symbol
of $T_N$, then $a_N\in \mascS (\rr {2d})$, and by straight-forward
computations we get
$$
a_N = \sum _j \lambda _j^{1/2N}W_{f_j,f_j}.
$$
By \cite[Lemma 4.1.2]{Toft0} we get
$$
\sum _j \lambda _j^{1/2N}\nm {H^Nf_j}{L^2}^2<\infty 
$$
for every $N\ge 0$, and the result follows in the case $p>0$ and $q=2$.

\par

Next assume that $q\in (0,2)$, and let $\omega _r =\eabs \cdo ^r
\in \mascP (\rr {2d})$,
$q_0=(2q)/(2-q)$ and $N_0>dq_0$ be an integer. Then
$\omega _{-2N_0}\in L^{q_0}(\rr {2d})$, and H{\"o}lder's inequality
gives
\begin{equation*}
\nm {H^Nf}{M^q}\lesssim \nm {\omega _{-2N_0}}{L^{q_0}}
\nm {H^Nf}{M^2_{(\omega _{2N_0})}} \asymp \nm {H^{N+N_0}f}{L^2},
\end{equation*}
when $f\in \mascS (\rr d)$, and the result follows in this case from the
case $q=2$.

\par

If instead $q\ge2$, then $\nm {H^Nf}{M^q}\lesssim \nm {H^Nf}{L^2}$
for every admissible $f$, and the result again follows from the case
when $q=2$.

\par

The assertion therefore follows if in the case
$\op _t(a)\ge 0$.

\par

Finally assume that $a\in \mascS (\rr {2d})$ is general. The Weyl symbol
of the operators $T^*\circ T$ and $T\circ T^*$ 
are given by
$$
b=\sum \lambda _j^2W_{f_j,f_j}\quad \text{and}\quad
c=\sum \lambda _j^2W_{g_j,g_j},
$$
respectively, and belong to $\mascS (\rr {2d})$, in view of Proposition
\ref{SchwartzKernels}. Since $T^*\circ T$
and $T\circ T^*$ are positive semi-definite, it follows from the first
part of the proof that
$$
\sum _j \lambda _j^p\nm {H^Nf_j}{M^q}^2<\infty 
\quad \text{and}\quad
\sum _j \lambda _j^p\nm {H^Ng_j}{M^q}^2<\infty 
$$ 
hold for every $p,q\in (0,\infty ]$ and $N\ge 0$. The estimate
\eqref{SymbSpecExpEst} now follows from these estimates
and Cauchy-Schwartz inequality.
\end{proof}

\par

We also need the following result related to Theorem 3.1 in \cite{GaSa}.

\par

\begin{lemma}\label{STFTWindows}
Let $p\in (0,2]$ and let $\omega ,\omega _1,\omega _2\in \mascP
_E(\rr {2d})$ be such that
$$
\omega (X_1-X_2)\lesssim \omega _1(X_1)\omega _2(X_2),\qquad X_1,X_2\in \rr {2d}.
$$
Then the map $(f,\phi )\mapsto V_\phi f$ is continuous from $M^p_{(\omega _1)}(\rr d)
\times M^p_{(\omega _2)}(\rr d)$ to $L^p_{(\omega )}(\rr {2d})$, and
\begin{equation}\label{NormsEsts}
\nm {V_\phi f}{L^p_{(\omega )}} \le C \nm f{M^p_{(\omega _1)}}
\nm \phi{M^p_{(\omega _2)}},
\end{equation}
where the constant $C$ is independent of $f\in M^p_{(\omega _1)}(\rr d)$ and
$\phi \in M^p_{(\omega _2)}(\rr d)$.
\end{lemma}

\par

\begin{proof}
First assume that $p\le 1$. Let $\psi \in \Sigma _1(\rr d)$, $\ep >0$,
$\Lambda =\ep \zz {2d}$, $\{ c(\mabfj )\} _{\mabfj \in \Lambda}\in \ell
^p_{(\omega _1)}(\Lambda )$ and $\{ d(\mabfk )\} _{\mabfk \in \Lambda}
\in \ell ^p_{(\omega _2)}(\Lambda )$ be chosen such that
$$
f(x) = \sum _{j,\iota \in \ep \zz d}c(j,\iota )\psi (x-j)e^{i\scal x\iota}
\quad \text{and}\quad
\phi (x) = \sum _{k,\kappa \in \ep \zz d}d(k,\kappa )\psi (x-k)e^{i\scal x\kappa}
$$
(cf. Proposition \ref{ConseqThmS}).

\par

By straight-forward computations it follows that
$$
V_\phi f(X) = \sum _{\mabfj ,\mabfk \in \lambda} c(\mabfj )\overline{d(\mabfk )}
\Psi (X+\mabfk -\mabfj) R_{\mabfj ,\mabfk}(X),
$$
where $\Psi =V_\psi \psi \in \Sigma _1(\rr {2d})$ and $R_{\mabfj ,\mabfk}$
is a function of exponential type such that $|R_{\mabfj ,\mabfk}| =1$,
for every $\mabfj$ and $\mabfk$. This gives
\begin{multline*}
\nm {V_\phi f}{L^p_{(\omega )}} ^p
\le
\sum _{\mabfj ,\mabfk \in \Lambda} |c(\mabfj )|^p |d(\mabfk )|^p
\nm {\Psi (\cdo +\mabfk -\mabfj)\omega }{L^p}^p
\\[1ex]
\le
\sum _{\mabfj ,\mabfk \in \Lambda} |c(\mabfj )|^p |d(\mabfk )|^p
\nm {\Psi v }{L^p}^p\omega (\mabfj -\mabfk)^p
\\[1ex]
\lesssim
\sum _{\mabfj ,\mabfk \in \Lambda} |c(\mabfj )\omega _1(\mabfj)|^p
|d(\mabfk )\omega _2(\mabfj)|^p
\asymp \nm f{M^p_{(\omega _1)}}\nm \phi{M^p_{(\omega _2)}},
\end{multline*}
when $v\in \mascP _E(\rr {2d})$ is chosen such that $\omega$ is $v$-moderate. Here
the first inequality follows from the fact that $p\le 1$ and the last inequality follows
from the assumptions. This gives the result in the case $p\le 1$.

\par

A slight modification of the proof of (2.8) in \cite{Toft11} gives the result in the
remaining case where $p\in [1,2]$. The details are left for the reader.
\end{proof}

\par

The next result concerns extensions of certain convolution relations in \cite{Toft0,Toft1}
between Schatten-von Neumann symbols and Lebesgue spaces to the case when the
Lebesgue parameters are allowed to be smaller than $1$.

\par

\begin{prop}\label{SchattenConv}
Let $p\in (0,1]$ and $t=1/2$. Then the map $(a,b)\mapsto a*b$ is continuous from
$s^p_{t,p}(\rr {2d})\times s^p_{t,p}(\rr {2d})$ to $L^p(\rr {2d})$.
\end{prop}

\par

\begin{proof}
Let $a,b\in s^p_{t,p}$. Since $s^p_{t,p}\subseteq s_{t,1}$, it follows that $a*b$ is well-defined
and belongs to $L^1$, in view of Theorem 2.1 in \cite{Toft1}.

\par

Now let
\begin{gather*}
\{ \lambda _j \} _{j=0}^\infty \in \ell ^p (\mathbf N),
\quad
\{ \mu _j \} _{j=0}^\infty \in\ell ^p (\mathbf N),\quad \{ f_{l,j} \} _{j=0}^\infty \in \ON (L^2(\rr d))
\intertext{and}
\{ g_{l,j} \} _{j=0}^\infty \in \ON (L^2(\rr d)),\quad l=1,2,
\end{gather*}
be such that $\lambda _j\ge 0$ and $\mu _j\ge 0$ for every $j\ge 0$,
$$
\sup _{j,l}\nm {f_{l,j}}{M^{2p}}<\infty
\quad \text{and}\quad
\sup _{j,l}\nm {g_{l,j}}{M^{2p}}<\infty ,
$$
and
$$
a=\sum _{j=0}^\infty \lambda _j W_{f_{1,j},f_{2,j}}
\quad \text{and}\quad
b=\sum _{j=0}^\infty \mu _j W_{g_{1,j},g_{2,j}} .
$$
Then
\begin{multline}\label{convLebSchattEst}
\nm {a*b}{L^p}^p = \int \left | \sum _{j,k} \lambda _j\mu _k
(W_{f_{1,j},f_{2,j}}*W_{g_{1,k},g_{2,k}} )(X)\right |^p \, dX
\\[1ex]
\le
\sum _{j,k} \lambda _j^p\mu _k^p
\nm {W_{f_{1,j},f_{2,j}}*W_{g_{1,k},g_{2,k}}}{L^p}^p.
\end{multline}

\par

By straight-forward computations we get
$$
|(W_{f_{1,j},f_{2,j}}*W_{g_{1,k},g_{2,k}} ) (x,\xi )|
=
C|(V_{\check f_{2,j}}g_{1,k})(x,\xi )(V_{\check f_{1,j}}g_{2,k})(x,\xi )|,
$$
for some constant $C$. Hence
Cauchy-Schwartz inequality and Lemma \ref{STFTWindows} give
\begin{multline*}
\nm {W_{f_{1,j},f_{2,j}}*W_{g_{1,k},g_{2,k}}}{L^p} =
\nm {V_{\check f_{2,j}}g_{1,k}\cdot V_{\check f_{1,j}}g_{2,k}}{L^p}
\\[1ex]
\le
\nm {V_{\check f_{2,j}}g_{1,k}}{L^{2p}}\nm {V_{\check f_{1,j}}g_{2,k}}{L^{2p}}
\lesssim \nm {f_{1,j}}{M^{2p}}\nm {f_{2,j}}{M^{2p}}\nm {g_{1,k}}{M^{2p}}\nm {g_{2,k}}{M^{2p}}.
\end{multline*}

\par

By inserting this into \eqref{convLebSchattEst} we get
\begin{multline*}
\nm {a*b}{L^p} \lesssim \left ( \sum _{j,k} \lambda _j^p\mu _k^p
\nm {f_{1,j}}{M^{2p}}^p \nm {f_{2,j}}{M^{2p}}^p
\nm {g_{1,k}}{M^{2p}}^p \nm {g_{2,k}}{M^{2p}}^p \right )^{1/p}
\\[1ex]
= \nm a{s_p^p}\nm b{s_p^p},
\end{multline*}
and the result follows.
\end{proof}

\par

\begin{proof}[Proof of Theorem \ref{CompSuppSchatten}]
The equality and the last embedding in \eqref{CompSuppSchattenEq}
follow from \cite[Proposition 4.3]{Toft12} and Theorem \ref{thmOpSchatten}.
The first embedding in \eqref{CompSuppSchattenEq} follows from Corollary
2.12 in \cite{Toft1} in the case $p\ge 1$.
It remains to prove the first embedding in \eqref{CompSuppSchattenEq}
in the case $p<1$.

\par

Therefore, assume that $p<1$, let $a\in \mascE '(\rr {2d})$ and choose
$\fy \in C_0^\infty (\rr {2d})$ such that $\fy =1$ on $\supp a$. Then $\mascF _\sigma
\fy \in s_{t,p}^p(\rr {2d})$ by Proposition \ref{SchattenExp}. Hence Lemma
\ref{basicLemmaspp} and Proposition \ref{SchattenConv} give
\begin{multline*}
\nm a{\mascF L^p} = \nm {\fy a}{\mascF L^p} \asymp \nm {(\mascF _\sigma \fy )
* (\mascF _\sigma a)}{L^p}
\\[1ex]
\lesssim \nm {\mascF _\sigma \fy }{s_{t,p}^p} \nm {\mascF _\sigma a}{s_{t,p}^p}
\lesssim \nm a{s_{t,p}^p},
\end{multline*}
which gives the result.
\end{proof}

\medspace

\end{document}